%% file: sampzeroarxiv.tex
\documentclass[twoside,11pt]{article}
\usepackage{amsmath,color}
\usepackage{algorithm,bm}
\usepackage{enumerate, enumitem}
\setlist{leftmargin=5.5mm}

\usepackage{algpseudocode}
\newenvironment{proof}{\par\noindent{\bf Proof\ }}{\hfill\BlackBox\\[2mm]}
\usepackage{graphicx}
\usepackage{epstopdf}
\usepackage{amsmath,authblk}
\usepackage{amsfonts}
\usepackage[square,sort,comma,numbers]{natbib}
\newtheorem{remark}{Remark}

\usepackage{hyperref}

\usepackage{amsmath}    
\usepackage[table]{xcolor}
\usepackage[T1]{fontenc}
\usepackage[utf8]{inputenc}

\usepackage{epstopdf}
\usepackage{amsmath}
\usepackage{epsf}
\usepackage{enumitem}
\usepackage{enumerate,mathtools}
\usepackage{subcaption}

\allowdisplaybreaks

\DeclareMathOperator*{\argmin}{argmin}

\newcommand{\R}{\mathbb{R}}

\newtheorem{lemma}{Lemma}[section]

\newtheorem{theorem}{Theorem}[section]
\newtheorem{proposition}{Proposition}[section]
\newtheorem{assumption}[theorem]{Assumption}

\newcommand{\defeq}{\stackrel{def}{=}}

\def\E{{\bf E}}
\def\exp{{\rm exp}}

\def\hx{{\hat x}}

\def\nh{{n+\frac{1}{2}}}
\def\ev2{{\expec{\|v_n\|^2}}}
\def\ef2{{\expec{\|\nabla f(x_n)\|^2}}}

\newcommand{\expec}[1]{\mathbf{E}\left[ #1 \right] }
\newcommand{\expect}[2]{\mathbf{E}_{#1}\left[ #2 \right] }
\newcommand{\Ea}[1]{\mathbf{E}_\alpha#1  }
\newcommand\numberthis{\addtocounter{equation}{1}\tag{\theequation}}

\newcommand{\pd}[2]{\frac{\partial #1}{\partial #2}}

\title{Stochastic Zeroth-order Discretizations of Langevin Diffusions for Bayesian Inference}
\author[1]{Abhishek Roy\thanks{abroy@ucdavis.edu}}
\author[2]{ Lingqing Shen\thanks{lingqins@andrew.cmu.edu. Work done while visiting UC Davis as an exchange student.}}
\author[1]{Krishnakumar Balasubramanian\thanks{kbala@ucdavis.edu}}
\author[3]{Saeed Ghadimi\thanks{sghadimi@uwaterloo.ca}}
\affil[1]{Department of Statistics, University of California, Davis}
\affil[2]{Tepper School of Business, Carnegie Mellon University}
\affil[3]{Department of Management Sciences, University of Waterloo}

\begin{document}
\maketitle

\begin{abstract}
Discretizations of Langevin diffusions provide a powerful method for sampling and Bayesian inference. However, such discretizations require evaluation of the gradient of the potential function. In several real-world scenarios, obtaining gradient evaluations might either be computationally expensive, or simply impossible. In this work, we propose and analyze stochastic zeroth-order sampling algorithms for discretizing overdamped and underdamped Langevin diffusions. Our approach is based on estimating the gradients, based on Gaussian Stein's identities, widely used in the stochastic optimization literature. We provide a comprehensive sample complexity analysis -- number noisy function evaluations to be made to obtain an $\epsilon$-approximate sample in Wasserstein distance -- of stochastic zeroth-order discretizations of both overdamped and underdamped Langevin diffusions, under various noise models. We also propose a variable selection technique based on zeroth-order gradient estimates and establish its theoretical guarantees. Our theoretical contributions extend the practical applicability of sampling algorithms to the noisy black-box and high-dimensional settings. 
\end{abstract}

\input{intro}

\input{zerolmc_SLC}

\input{zo_noise}

\input{hdzlmc}

\section{Discussion}
In this work, we proposed and analyzed zeroth-order discretizations of overdamped or underdamped Langevin diffusions. We provide a through analysis of the oracle complexity of such sampling algorithms under various noise models on the zeroth-order oracle and provide simulation results corroborating the theory. Recall that our zeroth-order gradient estimators used in this work were based on Gaussian Stein's identity and could be used for the case when $f$ is defined on the entire Euclidean space $\mathbb{R}^d$.In several situation, for example, in sampling from densities with compact support~\citep{brosse2017sampling, bubeck2018sampling} and in computing volume of convex body~\citep{brazitikos2014geometry}, one needs to compute the gradient of the function (and density) supported on $\mathcal{M} \subset \mathbb{R}^d$. For these situations, one can use a version of Stein's identity based on score functions to compute the gradient and Hessian. To explain more, we first recall some definitions. The score function $S_{p} \colon \mathcal{M} \rightarrow \mathbb{R}^d$ associated to density $p(u)$ defined over $\mathcal{M}$ is defined as  
\begin{align*}
S_p(u) = - \nabla_u  [\log p(u) ] = - \nabla _u p(u) / p(u).
\end{align*}
In the above definition,  the derivative is  taken with respect to the argument $u$ and not the parameters of the density $p(u)$. Based on the above definition, we have the following versions of Stein's identity; see, for example,~\citep{gorham2015measuring}.
\begin{proposition}\label{prop:1stein}
Let $U$ be a  $\mathcal{M}$-valued random vector with density $p(u)$. Assume that $p\colon \mathcal{M} \rightarrow \mathbb{R}$ is differentiable. In addition, let  $g : \mathcal{M} \to \mathbb{R}$ be a continuous function such that  $\E_U[\nabla g (U)]$ exists and the following is true: $\int_{u \in \mathcal{M}} \nabla_u \left( g(u) p(u)\right) du = 0$. Then it holds that 
$$ \E_U[ g(U) \cdot S(U) ]  = \E_U [ \nabla g(U) ] ,$$
where $S(u) = -  \nabla p(u) / p(u)$ is the score function of $p(u)$.
\end{proposition}
In order to leverage the above identities to estimate the gradient of a given function $f(\theta): \mathcal{M} \to \mathbb{R}$, consider $g(U) = f(\theta +U)$ where $U\sim p(u)$ is a $\mathcal{M}$-valued random variable and appeal to the above Stein's identity above, as done in Section~\ref{sec:zolmc} for with Gaussian random variables.  A special case of the above idea, when the space $\mathcal{M}$ is a Riemannian sub-manifold embedded in an Euclidean space was considered in~\cite{li2020zeroth} in context of stochastic zeroth-order Riemannian optimization. We postpone a rigorous analysis of the estimation and approximation rates in the general setting, and their applications to black-box sampling on non-Euclidean spaces for future work.

\bibliographystyle{alpha}
\bibliography{sampzeroarxiv}
\input{appendix}

\end{document}

%% file: intro.tex
\section{Introduction}\label{sec:intro}

First generation sampling algorithms, for example, Metropolis-Hastings algorithm are oblivious to the geometry of the target density as a result of which they suffer from slower rates of convergence. However, they are efficiently implementable and widely applicable, as they are based only on exact density function evaluations; see, for example,~\cite{belisle1993hit, kannan1995isoperimetric, mengersen1996rates,lovasz1990mixing, lovasz2007geometry, dunson2020hastings, martin2020computing}, for more details about such algorithms. Motivated by statistical physics principles, various researchers developed second-generation of sampling algorithms, that leverage geometric information regarding the target density~\citep{roberts1998optimal,neal2011mcmc, roberts1996exponential,stramer1999langevina, stramer1999langevinb, girolami2011riemann, barp2018geometry}. Such algorithms are based on gradient-based discretizations of continuous-time underdamped or overdamped Langevin diffusions. Although such algorithms were developed much earlier, recently strong theoretical guarantees have been established for sampling in the works of~\cite{durmus2017nonasymptotic, durmus2016high,dalalyan2017theoretical, dalalyan2017user,cheng2017underdamped, cheng2018sharp, dwivedi2018log} and several others. Such algorithms typically perform empirically better and exhibit faster rates of convergence compared to the first generation sampling algorithms mentioned above.

In this work, given a density function $\pi:\mathbb{R}^d \to \mathbb{R}$, with potential function $f:\mathbb{R}^d \to \mathbb{R}$, of the form
\begin{align}\label{eq:sampling}
\pi(\theta) = \frac{e^{-f(\theta)}}{\int_{\mathbb{R}^d} e^{-f(r)} \, d r }
\end{align}
we consider the problem of sampling when we only have access to noisy evaluations of the potential function $f$. We refer to this problem as \emph{stochastic zeroth-order sampling}. Our approach is based on discretizing overdamped and underdamped Langevin diffusions using stochastic zeroth-order oracles, which, when queried returns noisy unbiased evaluations, $F(x,\xi)$, of the function value $f(x)$. That is, we have $\E[F(x,\xi)] =f(x)$, where $\xi$ is the random noise in our function evaluations, which is not necessarily an additive noise. Our motivations for studying such problems are three-fold: 
\begin{itemize}[noitemsep,leftmargin=0.1in]
\item \textbf{Computationally Complexity of Gradient-evaluation}:  A majority of existing discretizations of Langevin diffusions require computing the gradient of the potential function $f$ in each iteration. It is well-know that for a wide class of functions which could be expressed based on compositions of elementary differentiable functions, the computational cost of evaluating the gradient is $4$ to $5$ times more than that of evaluating the function; see, for example~\cite{griewank2008evaluating}. Furthermore, in order to compute the gradient, it is necessary to store several intermediate gradients, which increases the memory requirement.  Hence, for several potential functions, Langevin-discretization based sampling algorithms might end up spending more time and memory for computing and storing gradients in each iteration. To reduce the wall-clock runtimes of such sampling algorithms, it is of interest to develop discretization of Langevin diffusions based only on function evaluations. 

\item \textbf{Non-availability of Analytic form of Potential Function:} In a variety of scientific problems, the potential function $f$ might not even be available in closed form, either due to the sheer size of the dataset (see, for example~\cite{sherlock2015efficiency}), or due to the constraints in the physical process underlying the statistical model (see, for example~\cite{beaumont2003estimation, golightly2011bayesian,knape2012fitting}). In these situations, we do not have access to the analytical from of true potential function, let alone its gradients, which are required for discretizing Langevin diffusions. Hence, it is of great interest to develop discretization of Langevin diffusions based on noisy function evaluations to widen the applicability of Bayesian inference. It is worth mentioning here that, in the case of Metropolis-Hastings algorithms,~\cite{andrieu2009pseudo, sherlock2015efficiency} developed and analyzed the so-called Pseudo-Marginal Metropolis-Hasting algorithms which work with unbiased noisy density evaluations. However, similar algorithms for sampling based on discretizing Langevin diffusions are lacking in the literature, except for the recent work on~\cite{alenlov2016pseudo} which considered a pseudo-marginal Hamiltonian Monte Carlo algorithm. Our second motivation for this work is to fill this gap and to develop and analyze a unified framework for stochastic zeroth-order discretization of Langevin diffusions for Bayesian inference.

\item \textbf{Automating Bayesian Inference:} From a practitioner's perspective, statistical modeling is an inherently iterative process. The probabilistic model is typically refined during the scientific process based on the fit to the data. In the context of sampling, this process could be understood as changing the potential function $f$ in the modeling process. However, each time the function $f$ is changed, it is also invariably required to re-code the sampling algorithm based on the analytically computed gradient of the function $f$ under consideration. Our third motivation in this work is to automate this process, to help the practitioner with quick experimentation. As we will see later, our proposed methodology allows for sampling from a wide variety of density functions in a unified manner, as long as we have an oracle to obtain (noisy) evaluations of the potential function $f$. It is worth mentioning that recently~\cite{ranganath2014black,ranganath2015deep,kucukelbir2017automatic} developed related automated Bayesian inference algorithms based on variational inference.  
\end{itemize}

\subsection{Preliminaries}
Consider the continuous-time Langevin diffusion process $\{ L_T: T \in \mathbb{R}_+\}$ given by the following  stochastic differential equation, 
\begin{align}\label{eq:ctsde}
d L_T = - \nabla f(L_T) dT + \sqrt{2} d W_T,
\end{align}
where $T \in \mathbb{R}_+$ and $\{W_T : T \in \mathbb{R}_+\}$ is a $d$-dimensional Brownian motion and $\nabla f(\theta) \in \mathbb{R}^d$ denotes the gradient of $f(\theta)$. The Euler-Maruyama discretization of the process in~\eqref{eq:ctsde} is given by the following Markov chain:
\begin{align}\label{eq:lmc}
x_{n+1} = x_{n} - h_{n+1} ~\nabla f(x_{n}) + \sqrt{2 h_{n+1}} \varepsilon_{n+1},
\end{align}
for the discrete time index $n = 0, 1, 2\ldots$. Here $\varepsilon_{n} \in \mathbb{R}^d$ is a sequence of independent standard Gaussian vectors, $h_n $ denotes the step-size and an initial point $ x_{0}$
is assumed to be given. The above discretization is called as the Langevin Monte Carlo (LMC) sampling algorithm. The update step of the LMC sampling algorithm shares similarity with the standard gradient descent algorithm from the optimization literature. As a prelude to the rest of the paper, our main idea in this work is to provide a non-asymptotic analysis of using stochastic zeroth-order gradient estimators (described in details in Section~\ref{sec:mainmethod}) in place of the true gradient in~\eqref{eq:lmc} and related discretizations.  

Denoting the distribution of the random vector $x_{n}$ by $\varpi_n$, to evaluate the performance of the sampling algorithm, the 2-Wasserstein distance between $\varpi_n$ and the target density $\pi(\theta)$ is considered. For measures, $p$ and $q$ defined on $(\mathbb{R}^d, \mathcal{B}(\mathbb{R}^d))$, the 2-Wasserstein distance is defined as:
\begin{align}
W_2(p,q) := \left(\underset{\varrho \in \varrho(p,q)}{\inf} \int_{\mathbb{R}^d \times \mathbb{R}^d}  \| \theta -\theta'\|_2^2 \, d\varrho(\theta,\theta') \right)^{1/2},
\end{align}
where $ \varrho(p,q)$ is the set of joint distribution that has $p$ and $q$ as its marginals. The performance of the sampling updates is measured by the above 2-Wasserstein distance between the distribution $\varpi_n$ and the target density $\pi$, i.e., $W_2(\varpi_n,\pi)$. Specifically, the \emph{iteration complexity} of the algorithm is defined as the  number of iterations $N$, required to get $W_2(\varpi_N,\pi) \leq \epsilon$. We also define the notion of \emph{oracle complexity} which is the number of calls to the first-order or stochastic zeroth-order oracle used to obtain $W_2(\varpi_N,\pi) \leq \epsilon$. For the LMC algorithm in~\eqref{eq:lmc}, as we use only one gradient evaluation in each iteration, the oracle and iteration complexity becomes the same. 

In order to obtain theoretical guarantees, a common assumption made in the literature on LMC is that the function $f$ is smooth and strongly convex.
\begin{assumption}\label{smooth_assum}
Letting $\|\cdot  \| = \| \cdot \|_2$ denote the Euclidean norm on $\mathbb{R}^d$, the potential function $f$,
\begin{enumerate}[leftmargin=0.85cm]
\item[\textbf{A1}:] is strongly convex i.e., $f(\theta) - f(\theta') - \nabla f(\theta')^\top (\theta - \theta') \geq \frac{m}{2} \| \theta - \theta' \|^2$, for $m >0$.
\item[\textbf{A2}:] has Lipschitz continuous gradient, i.e., $\|\nabla f(\theta) -\nabla f(\theta')\| \le M\|\theta-\theta'\|$ for $M >0$.
\end{enumerate}
\end{assumption}
The above assumptions on the potential function in-turn makes the density function $\pi$ strongly log-concave and smooth. Such an assumption is satisfied in several sampling and Bayesian inference problems including sampling from mixture of Gaussian distributions and Bayesian logistic regression. Further assuming access to certain inaccurate gradients, \cite{dalalyan2017user} provide theoretical guarantees for sampling under Assumption~\ref{smooth_assum}. Specifically, instead of the true gradient $\nabla f(x_{n})$ in each step, it is assumed that we observe $g_{n} = g(x_{n}) = \nabla f(x_{n}) + \zeta_n$, for a sequence of random noise vectors $\zeta_n$ that satisfies certain bias and variance assumption. Then, the noisy LMC updates corresponds to the case of the updates in Equation~\ref{eq:lmc}, with $\nabla f(x_{n})$ replaced by $g_{n}$. For such an update,~\cite{dalalyan2017user} have the following non-asymptotic result. Before providing the result, we remark that due to the assumptions on the stochastic gradient made in~\eqref{sfo_assumption}, this setting is referred to as stochastic first-order setting.

\begin{theorem}\citep{dalalyan2017user}\label{thm:startingpoint}
Assume that the bias and variance of $\zeta_n$ satisfies respectively, for all $n=1, 2, \ldots$,
\begin{align}\label{sfo_assumption}
\E[\lVert\E(\zeta_n|x_{n})\rVert^2] \leq\delta_b^2d \qquad~\text{and}\qquad
 \E[\lVert\zeta_n-\E(\zeta_n|x_{n})\rVert^2] \leq\delta_v^2d.
 \end{align}
 Let the function $f$ satisfy Assumption~\ref{smooth_assum}. If $h \leq 2/(m+M)$, the following result holds true.
\begin{align*}
W_2(\varpi_n,\pi) \leq (1-mh)^n W_2(\varpi_0,\pi) + 1.65\frac Mm (hd)^{1/2} + \frac{\delta_b \sqrt{d}}{m} + \frac{\delta_v^2 (hd)^{1/2}}{1.65M+\sigma \sqrt{m}}.
\end{align*}
\end{theorem}

\begin{remark}
More generally, if the bounded bias and variance condition are changed to
\begin{align*}
 \E[\lVert\E(\zeta_n|x_{n})\rVert^2] \leq\delta_b^2d^\alpha \qquad~\text{and}~\qquad  \E[\lVert\zeta_n-\E(\zeta_n|x_{n})\rVert^2]\leq\delta_v^2d^\beta,
 \end{align*}
respectively, for some $\alpha, \beta > 0$, the conclusion turns into
\begin{align*}
W_2(\varpi_n,\pi)\leq &(1-mh)^n W_2(\varpi_0,\pi)+\frac{1.65M(hd)^{1\slash2}}{m}+\frac{\delta_b d^{\alpha\slash2}}{m} \\ &+\frac{\delta_v^2hd^\beta}{1.65 M(hd)^{1\slash2}+\delta_b d^{\alpha\slash2}+\delta_v(mh)^{1\slash2}d^{\beta\slash2}}. \end{align*}
Furthermore, in the case that $\beta>\max\{1,\alpha\}$, the last term is dominated by $d^{\beta\slash2}$.
\end{remark}

\subsection{The Zeroth-order Methodology}\label{sec:mainmethod}

The use of zeroth-order information (i.e., noisy function evaluations) for optimizing a function goes back to the works of~\cite{kiefer1952stochastic, blum1954multidimensional}, that used stochastic version of finite-difference gradient approximation methods for estimating the maximum of a regression function (or equivalently mode of a density function). Since, then zeroth-order optimization has developed into an independent field in itself; see, for example~\cite{spall2005introduction, conn2009introduction, audet2017derivative, larson2019derivative} for an more up-to-date account of this field. More recently, the focus has been more on developing a non-asymptotic understanding of stochastic zeroth-order optimization~\cite{ghadimi2013stochastic, duchi2015optimal, nesterov2017random, balasubramanian2018zeroth}. Despite the fact that stochastic zeroth-order optimization is a well-developed field, to the best of our knowledge, there is no prior work on using related techniques for the closely related problem of zeroth-order discretizations of Langevin diffusions; specifically in terms of non-asymptotic analysis.

We now describe the precise assumption made on the \emph{stochastic zeroth-order oracle} in the first part of this work. 
\begin{assumption}\label{as:zo}
	For any $\theta\in\mathbb{R}^d$, the stochastic zeroth-order oracle outputs an estimator $F(\theta,\xi)$ of $f(\theta)$ such that, $\expec{F(\theta,\xi)}=f(\theta)$, $\expec{\nabla F(\theta,\xi)}=\nabla f(\theta)$, and $\expec{\|\nabla F(\theta,\xi)-\nabla f(\theta)\|^2}\leq \sigma^2$.
\end{assumption}
The assumption above assumes that we have accesses to a stochastic zeroth-order oracle which provides unbiased function evaluations with bounded variance. It is worth noting that in the above, we do not necessarily assume the noise $\xi$ is additive. Our gradient estimator is then constructed by leverage the Gaussian smoothing technique~\citep{nesterov2017random, ghadimi2013stochastic,balasubramanian2018zeroth}, which is amenable for fine-grained non-asymptotic analysis. Specifically, for a point $\theta \in \mathbb{R}^d$, we define an estimate $g_{\nu,b}(\theta)$, of  the gradient $\nabla f(\theta) $ as follows:
\begin{align}\label{eq:gradest}
g_{\nu,b}(\theta) = \frac{1}{b} \sum_{i=1}^b \frac{F(\theta+\nu u_i,\xi_i) - F(\theta,\xi_i) }{\nu} u_i
\end{align}
where $u_i \sim N(0,I_d)$ and are assumed to be independent and identically distributed. An interpretation of the gradient estimator in~\eqref{eq:gradest} as a consequence of Gaussian Stein's identity, popular in the statistics literature~\cite{stein1972bound}, was provided in~\cite{balasubramanian2018zeroth}. Finally, the parameter $b$ is called as the batch-size parameter. It turns out that in the stochastic zeroth-order setting invariably we require $b>1$, which in turn leads to the (zeroth-order) oracle complexity being an order $b$ times that of iteration complexity. In Section~\ref{sec:zolmc} and~\ref{as:lsi}, we use the above gradient estimation technique in the context of discretizing overdamped and underdamped Langevin diffusion and develop their oracle and iteration complexities. In order to establish the results, we will use the following Lemma due to \cite{balasubramanian2018zeroth} which provides an upper bound on the variance of $g_{\nu,b}$. 
\begin{lemma}\label{lm:zogradvar}\cite{balasubramanian2018zeroth}
	Let $g_{\nu,b}$ be defined as in \eqref{eq:gradest}. Then under Assumption~\ref{as:zo}, and condition A1 of Assumption~\ref{smooth_assum}, we have,
	\begin{align*}
	&\expec{\|g_{\nu,b}(\theta)-\nabla f_\nu(\theta)\|^2}\leq\frac{2(d+5)(\|\nabla f(\theta)\|^2+\sigma^2)}{b}+\frac{\nu^2 M^2(d+3)^3}{2b} \numberthis\label{eq:ggradfnucon}\\
	&\expec{\|g_{\nu,b}(\theta)-\nabla f(\theta)\|^2}\leq \frac{4(d+5)(\|\nabla f(\theta)\|^2+\sigma^2)}{b}+\frac{3\nu^2 M^2(d+3)^3}{2}\numberthis\label{eq:ggradfcon}
	\end{align*}
\end{lemma}
where $f_\nu(\theta)=\mathbf{E}_u[f(\theta+\nu u)]$.
\vspace{0.1in}

 \textit{One-point versus two-point evaluation:} The gradient estimator in~\eqref{eq:gradest} is referred to as the two-point estimator in the literature. The reason is that, for a given random vector $\xi$, it is assumed that the stochastic function in~\eqref{eq:gradest} could be evaluated at two points,$F(\theta_1,\xi)$ and $F(\theta_2,\xi)$. Such an assumption is satisfied in several statistics, machine learning and simulation based optimization and sampling; see for example in~\cite{spall2005introduction, mokkadem2007companion, dippon2003accelerated, agarwal2010optimal, duchi2015optimal, ghadimi2013stochastic, nesterov2017random}. Yet another estimator is the one-point estimator which assumes that for each $\xi$, we observe only one noisy function evaluation $F(\theta,\xi)$. Admittedly, the one-point setting is more challenging than the two-point setting. Specifically, in the one-point feedback setting, Lemma~\ref{lm:zogradvar} no longer holds. From a theoretical point of view, the use of two-point evaluation based gradient estimator is primarily motivated by the sub-optimality (in terms of oracle complexity) of one-point feedback based stochastic zeroth-order optimization methods either in terms of the approximation accuracy or dimension dependency. 
 
 The use of one-point feedback for stochastic zeroth-order gradient estimation could be traced back to~\cite{nemyud:83}. Motivated by this, there has been several works in the machine learning community focusing on leveraging it for zeroth-order convex optimization. Specifically, considering the class of convex functions (without any further smoothness assumptions)  and adversarial noise (i.e., roughly speaking, with noise vectors not necessarily assumed to be independent and identically distributed (i.i.d.)),~\cite{bubeck2017kernel} proposed a polynomial-time algorithm and an oracle complexity of $\mathcal{O}(d^{21}/\epsilon^2)$. This was improved to $\mathcal{O}(d^5/\epsilon^2)$ recently in~\cite{lattimore2020improved}. Further assuming Lipschitz smooth convex functions,~\cite{belloni2015escaping} and~\cite{gasnikov2017stochastic}, in the i.i.d noise case, obtained an oracle complexity of $\mathcal{O}(d^{7.5}/\epsilon^2)$  and $\mathcal{O}(d/\epsilon^3)$ respectively. The best known lower bound in this case is known to be $\mathcal{O}(d^{2}/\epsilon^2)$, which was established by~\cite{shamir2013complexity}. Further assuming $(\beta-1)$ differentiable derivatives, for $\beta> 2$,~\cite{bach2016highly} obtained as oracle complexity of  $\mathcal{O}(d^{2}/\epsilon^{2\beta/(\beta-1)})$ and $\mathcal{O}(d^{2}/\epsilon^{(\beta+1)/(\beta-1)})$ respectively for the convex and strongly-convex setting, with i.i.d. noise case; see also~\cite{akhavan2020exploiting}. In contrast to the above discussion, with two-point feedback it is possible to obtain much improved oracle complexities (i.e., linear in dimension and optimal in $\epsilon$) for stochastic zeroth-order optimization, as illustrated in~\cite{nesterov2017random, ghadimi2013stochastic, duchi2015optimal,agarwal2010optimal}. Given this subtle differences between the two-point and one-point evaluation settings for stochastic zeroth-order gradient estimation, in  Section~\ref{sec:noise} we consider the effect of one-point gradient estimation technique for stochastic zeroth-order discretization of overdamped and underdamped Langevin diffusions.  

\subsection{Our Contributions}
Under the availability of the stochastic zeroth-order oracles, we make the following contributions to the literature on sampling.
\begin{enumerate}[noitemsep,leftmargin=0.14in]
\item We first consider the case of strongly log-concave and smooth densities and analyze a stochastic zeroth-order version of Euler-discretization of overdamped and underdamped Langevin diffusions, under the two-point feedback setting in Section~\ref{sec:zolmc}. For both cases, we characterize the oracle and iteration complexities to obtain $\epsilon$-approximate samples in term of $W_2$ metric.
\item We next consider in Section~\ref{sec:zormp}, a stochastic zeroth-order version of the recently proposed Randomized Midpoint Sampling method of the underdamped Langevin diffusion and characterize the oracle and iteration complexities to obtain $\epsilon$-approximate samples  in term of $W_2$ metric. We show that for certain range of $\epsilon$, this method achieves improved oracle complexity compared to the above method.
\item While the above contributions are for strongly log-concave densities, in Section~\ref{as:lsi}, we consider the more general class of densities satisfying log-Sobolev inequality and establish the oracle and iteration complexities of stochastic zeroth-order discretizations. 
\item While all of the above contributions use the two-point stochastic zeroth-order feedback setting, in Section~\ref{sec:noise}, we next consider the case of one-point feedback and characterize the corresponding oracle and iteration complexities for all the above discretizations.
 \item Next, in Section~\ref{sec:hdzlmc}, we consider variable selection for zeroth-order sampling. We specifically assume the unobserved function $f$ is sparse in the sense that it depends only on $s$ of the $d$ coordinates. We provide a variable selection method based on the estimated zeroth-order gradient, which in conjunction with the above discretizations reduces the oracle and iteration complexities to be only poly-logarithmically dependent on the dimensionality $d$ thereby enabling high-dimensional sampling.
\end{enumerate} 
Our contributions provide several theoretical insights on the performance of stochastic zeroth-order sampling algorithms, and widen the applicability of theoretically sound Bayesian inference to various practical situations where we do not know the analytical form of the potential function. All proofs are relegated to the appendix. 

%% file: zerolmc_SLC.tex
\section{Oracle Complexity Results under Strong Log-concavity}\label{sec:zolmc}
We now leverage the stochastic zeroth-order gradient estimation methodology introduced in Section~\ref{sec:mainmethod} for discretizing underdamped and overdamped Langevin diffusions. Throughout this section, we assume the target density is strongly log-concave and smooth (recall Assumption~\ref{smooth_assum}). 
\subsection{Zeroth-Order Langevin Monte Carlo}
 Replacing the true gradient in the first-order Langevin Monte Carlo algorithm in~\eqref{eq:lmc}, with the zeroth-order gradient estimation in~\eqref{eq:gradest}, we obtain the following Zeroth-Order LMC (ZO-LMC) algorithm:
\begin{align}
x_{n+1} = x_{n} - h  ~ g_{\nu,b}(x_{n}) + \sqrt{2 h} \varepsilon_{n+1} \numberthis\label{eq:zolmcupdate}
\end{align}
for $n = 0, 1, 2,\cdots,N-1$. Apart from the choice of step-size $h$, the ZO-LMC also requires two additional tuning parameters, the smoothing parameter $\nu$ and the batch-size $b$ of the zeroth-order gradient estimator $b$, that need to be set. Although ZO-LMC could be interpreted as a form of LMC with inaccurate gradient as in~\cite{dalalyan2017user}, the corresponding theoretical result from~\cite{dalalyan2017user} cannot be used directly for obtaining the oracle complexity of ZO-LMC, as the variance of the gradient in~\eqref{eq:gradest} is not bounded unless we make restrictive assumptions on the true gradient of $f$. We now state the main result of this section, which describes the oracle complexity of ZO-LMC. 
\begin{theorem}\label{thm:zlmc}
Let the potential function $f$ satisfy Assumption~\ref{smooth_assum}. Then, for the ZO-LMC algorithm in~\eqref{eq:zolmcupdate}, under Assumption~\ref{as:zo}, by choosing
\begin{align*}
 h=\frac{\epsilon^2}{d^2}, \quad b=\max(1,\sigma^2) d, \quad \nu=\frac{\epsilon}{\sqrt{d}}, \numberthis\label{eq:lmcparamchoice}
\end{align*}
we have $W_2(\varpi_N,\pi)\leq \epsilon$ for $0<\epsilon\leq \min\left(d\sqrt{\frac{2}{M+m}},\sqrt{\frac{m(d+5)}{8M^2}}\right)$, after
\begin{align*}
	N={\mathcal{O}}\left(\frac{ d}{\epsilon^2}\cdot\log\left(\frac{d}{\epsilon}\right)\right).  \numberthis\label{eq:LMCNbound}
\end{align*}
iterations. Hence, the total number of calls to the stochastic zeroth-order oracle is given by,
\begin{align*}
Nb={\mathcal{O}}\left(\frac{\max(1,\sigma^2)\cdot d^2}{\epsilon^2}\cdot\log\left(\frac{d}{\epsilon}\right)\right). \numberthis\label{eq:LMCoraclebound}
\end{align*}
\end{theorem}
\begin{remark}
Recall that for the exact gradient based LMC algorithm, to obtain $W_2(\varpi_N,\pi)\leq\epsilon$, we require  $N={\cal O}\left(d/\epsilon^2\cdot\log(d/\epsilon)\right)$ (see~\cite{dalalyan2017user}) which matches \eqref{eq:LMCNbound}. Thus, ZO-LMC matches the performance of LMC in terms of iteration complexity required to obtain $W_2(\varpi_N,\pi)\leq\epsilon$. However, in each iteration of the LMC algorithm, we only require one gradient evaluation. Hence, the total number calls to the first-order oracle is also given by ${\cal O}\left(d/\epsilon^2\cdot\log(d/\epsilon)\right)$. For the ZO-LMC, in contrast we require $b=d$ calls to the stochastic zeroth-order oracle in each iteration. Hence, the oracle complexity is given by~\eqref{eq:LMCoraclebound}. By a straight forward modification of the proof of Theorem~\ref{thm:zlmc}, for the ZO-LMC, if we restrict ourself to $b=1$, the iteration complexity increases to $N={\cal O}\left(d^2/\epsilon^2\cdot\log(d/\epsilon)\right)$, which will then also be the oracle complexity. Thus, the price we pay to match LMC in the absence of true gradient information is $O(d)$.  
\end{remark}

\begin{remark}
Recently~\cite{dwivedi2018log} analyzed the standard Metropolis Random Walk algorithm (MRW), which is a zeroth-order algorithm, in the non-noise setting. Specifically,~\cite{dwivedi2018log} showed that to achieve samples that are $\epsilon$-close to the target $\pi$ in total variation distance, MRW requires $\mathcal{O}(d^2 \log(1/\epsilon)$ calls to the non-noisy zeroth-order oracles. Considering the non-noisy setting, the result appears to seemingly have an exponential improvement in terms of $\epsilon$. However, this result was obtained under the so-called \emph{warm start} condition on the distribution of the initial vector $x_o$, which seems to be an opaque condition hiding the true complexity of the problem. For example, it is not clear how to pick such a warm start distribution for a given target $\pi$, in particular in the stochastic zeroth-order setting that we consider in this work. As a way to potentially avoid this opaque warm start condition,~\cite{dwivedi2018log} suggests to set $x_0 \sim N(x^*, \mathbf{I}_d)$, where   $x^*$ is the unique minimizer of $f(x)$ and $\mathbf{I}_d$ is the $d\times d$ identity covariance matrix. For this choice of initial vector, to obtain a sample which is $\epsilon$-close to the target $\pi$ in total variation distance,~\cite{dwivedi2018log} showed that MRW requires an oracle complexity of $\mathcal{O}(d^2 \log(1/\epsilon)$. However, in the zeroth-order setting, the oracle complexity of finding an $\epsilon$-minimizer of a strongly-convex and smooth function $f(x)$, is well-studied problem in stochastic optimization -- it is upper and lower bounded by $\mathcal{O}(d/\epsilon)$; see for example~\cite{duchi2015optimal, jamieson2012query, nesterov2017random, ghadimi2013stochastic}. This seems to negate the actual oracle complexity improvements shown in~\cite{dwivedi2018log}, as it really seems to require extremely careful initial distributions (i.e., knowledge of the exact minimzer), even in the non-noisy setting. Notwithstanding the fact that the results in~\cite{dwivedi2018log} are for the non-noisy setting, they are essentially no better than the oracle complexity results established for ZO-LMC algorithm in Theorem~\ref{thm:zlmc}, which also has the advantage that it does  not require any opaque warm start conditions or special initial distributions. 
\end{remark}

\subsection{Zeroth-Order Kinetic Langevin Monte Carlo}

In the previous section, we consider the stochastic zeroth-order discretizations of the overdamped Langevin diffusions. It is known that in the first-order setting, discretizations of underdamped Langevin diffusion obtain improved oracle complexities~\cite{dalalyan2018sampling, cheng2017underdamped}. Under Langevin diffusion process (also called as kinetic Langevin diffusion process) is given by the following stochasic differential equation:
\begin{align}\label{eq:kineticsde}
dV_T &= \left(\gamma V_T + \nabla f(L_T)\right) dT+ \sqrt{2\gamma} dW_T\\ \nonumber
dL_T &= V_T dT. 
\end{align}
where $\mathbf{I}_{d}$ is the $d \times d$ identity matrix. We refer the reader to~\cite{eberle2017couplings, cheng2017underdamped,dalalyan2018sampling} for more details about the above diffusion process and related theoretical results. Specifically, it was shown in~\cite{cheng2017underdamped, dalalyan2018sampling} that first-order discretizations of the kinetic diffusion process (referred to as KLMC in~\cite{cheng2017underdamped}) in~\eqref{eq:kineticsde} have better rates of convergence compared to similar first-order discretizations of the continous-process in~\eqref{eq:ctsde}. Specifically, recall that for the right choice of tuning parameters, LMC (i.e., first-order discretizations of~\eqref{eq:ctsde}) requires that $N =\mathcal{O}(d/\epsilon^2 \cdot \log(d/\epsilon))$ for $W_2(\varpi_N,\pi) \leq \epsilon$. Whereas, it was shown in~\cite{cheng2017underdamped, dalalyan2018sampling} $N = \mathcal{O}(\sqrt{d}/\epsilon \cdot \log(d/\epsilon))$ suffices (\cite{dalalyan2018sampling} provides a much sharper result compared to~\cite{cheng2017underdamped}). We emphasize that the above result does not immediately imply that KLMC might be the algorithm to use always (in comparison to LMC); indeed when considering also the dependence of the bound on the strong-convexity and smoothness parameters (though the condition number of the sampling density defined as $M/m$),~\cite{dalalyan2018sampling} precisely characterize when KLMC might be preferred over the vanilla LMC. The bottom line of their analysis is none of the method is uniformly better over the other method. 

The Euler-discretization of  the SDE in~\eqref{eq:kineticsde}, which is a first-order sampling algorithm is given by the following iterations:
\begin{align}\label{eq:KLMC}
\tilde{x}_{n+1} &=  \psi_0(h) \tilde{x}_{n} - \psi_1(h) \nabla f(x_{n}) + \sqrt{2\gamma} \tilde{\epsilon}_{n+1}\\ \nonumber
 x_{n+1}& = x_{n} + \psi_1(h) \tilde{x}_{n} - \psi_{2}(h) \nabla f(x_{n}) + \sqrt{2\gamma} \epsilon_{n+1} 
 \end{align}
where $(\tilde{\epsilon}_{n+1}, \epsilon_{n+1}) \in \mathbb{R}^{2d}$ is a a sequence of i.i.d standard Normal vectors, independent of $(\tilde{x}_0, x_0)$ and $\psi_0(t) = e^{-\gamma t}$ and $\psi_{n+1} = \int_0^T \psi_{n}(s) ds$. We refer to this algorithm as KLMC following the terminology of~\cite{dalalyan2018sampling}. Based on this, we now consider the ZO-KLMC updates as:
\begin{align}\label{eq:ZOKLMC}
\tilde{x}_{n+1} & = \psi_0(h) \tilde{x}_{n} - \psi_1(h) g_{\nu,b}(x_{n}) + \sqrt{2\gamma}  \tilde{\epsilon}_{n+1}\\ \nonumber
 x_{n+1} & = x_{n} + \psi_1(h) \tilde{x}_{n} - \psi_{2}(h) g_{\nu,b}(x_{n}) + \sqrt{2\gamma} \epsilon_{n+1} 
\end{align}
where $g_{\nu,b}$ is the zeroth-order gradient estimator as in~\eqref{eq:gradest}. In comparison to the ZO-LMC algorithm, the ZO-KLMC algorithm has an additional tuning parameter $\gamma$ that needs to be set. For the ZO-KLMC algorithm, we have the following complexity result. 

\begin{theorem}\label{thm:KLMCthms}
Let the potential function $f$ satisfy Assumption~\ref{smooth_assum}. If the initial point $(\tilde x_0,x_0)$ is chosen such that $\tilde x_0\sim N(0,\boldsymbol{I}_d)$, then, ensuring $\gamma\geq\sqrt{m+M}$, for the ZO-KLMC, under Assumption~\ref{as:zo}, by choosing,
\begin{align*}
h=\frac{m\epsilon}{12\gamma M\sqrt{d}},\quad \nu=\frac{\epsilon}{\sqrt{d}},\quad b=\frac{d^{1.5}\max(1,\sigma^2)}{\epsilon}, \numberthis\label{eq:klmcparchoice}
\end{align*}
we have $W_2(\varpi_N,\pi)\leq \epsilon$ for $0<\epsilon\leq \frac{12M\gamma^2\sqrt{d}}{m^2}$ after
\begin{align*}
N=\tilde{\mathcal{O}}\left(\frac{\sqrt{d}}{\epsilon}\right) \numberthis\label{eq:klmcNbound}
\end{align*}
iterations. Here $\tilde{\mathcal{O}}$ hides poly-logarithmic factors in $1/\epsilon$. Hence, the total number of  calls to the stochastic zeroth-order oracle is given by 
\begin{align*}
Nb=\tilde{\mathcal{O}}\left(\frac{d^{2}\max(1,\sigma^2)}{\epsilon^2}\right).\numberthis\label{eq:klmcoraclebound}
\end{align*}
\end{theorem}
\begin{remark}
We note that compared to ZO-LMC, while ZO-KLMC obtains improved iteration complexity, the iteration complexity still remains the same. The improvement in the iteration complexity is indeed a consequence of a similar improvement in the first-order setting as demonstrated in~\cite{cheng2017underdamped, dalalyan2018sampling}. 
\end{remark}

\subsection{Zeroth Order Randomized Midpoint Method}\label{sec:zormp}
Given that the ZO-KLMC offers no improvement over ZO-LMC in terms of oracle complexity despite its improved iteration complexity, it is worth examining if there are other discretizations that obtain improvements in oracle complexities. Towards that, in this section we analyze the zeroth-order version of the Randomized Mid-Point discretization of the underdamped Langevin diffusion, proposed in \cite{shen2019randomized}. In the first-order setting,~\cite{cao2020complexity} recently showed that the Randomized Mid-Point discretization of underdamped Langevin diffusion achieves the information theoretic lower bounds for sampling. See also~\cite{he2020ergodicity} for additional probabilistic results.

The crux of the randomized midpoint method is based on first representing the kinetic Langevin Monte Carlo in~\eqref{eq:kineticsde} in its integral format, and estimating the integrals based on a randomization technique. We also mention in passing that the randomized midpoint idea shares some similarities to symplectic integration methods~\cite{sanz1992symplectic} from the sampling literature and extragradient method~\cite{korpelevich1976extragradient} from the optimization literature, with the main difference being the randomized choice of step-size which leads to improved oracle complexities. We now provide the algorithm in the zeroth-order setting and the corresponding theoretical result. Let $\epsilon^{(i)}_n \in \mathbb{R}^d$, $i=1,2,3$, be a sequence of Gaussian random vectors generated according to the procedure described in Appendix A of~\cite{shen2019randomized}. Let $\alpha_n$ be a sequence of uniform random variables supported on the interval $[0,1]$. Then the zeroth-order Randomized Mid-Point Method (ZO-RMP) is given by the following updates: 
\begin{align}
&x_\nh=x_n+\frac{1-e^{-2\alpha_n h}}{2}v_n-\frac{u}{2}\left(\alpha_n h - \frac{1-e^{-2(\alpha_n h)}}{2}\right) g_{\nu,b}(x_n) + \sqrt{u} \epsilon^{(1)}_{n+1}   \label{eq:rmpupdatexnh}\\
&x_{n+1}=x_n+\frac{1-e^{-2h}}{2}v_n-\frac{uh}{2}(1-e^{-2(h-\alpha_n h)})g_{\nu,b}\left(x_\nh\right)+\sqrt{u} \epsilon^{(2)}_{n+1}\label{eq:rmpupdatexn1}\\
&v_{n+1}=v_ne^{-2h}-uhe^{-2(h-\alpha_n h)}g_{\nu,b}\left(x_\nh\right)+2\sqrt{u}\epsilon^{(3)}_{n+1}.\label{eq:rmpupdatevn1}
\end{align}
We remark that we use the same choice of batch-size, $b$, in~\eqref{eq:rmpupdatexnh},~\eqref{eq:rmpupdatexn1} and~\eqref{eq:rmpupdatevn1}, as using different batch sizes has no effect on the oracle complexity.
{\color{blue}
}  

\begin{theorem}\label{th:rmpmainresult}
Define $\kappa = M/m$ to be the condition number of the potential $f$ which satisfies Assumption~\ref{smooth_assum}. Furthermore, let the stochastic zeroth-order oracle satisfy Assumption~\ref{as:zo}.  Let $x^*$ be the minimizer of $f$, and $x_0$ be such that $\expec{f(x_{0})-f(x^*)}=O(d)$, and $v_0=0$. Then, for $0\leq \epsilon\leq 1$, by choosing,
	\begin{align}
	&h= C\min\left(\frac{(\epsilon\sqrt{m})^\frac{1}{3}}{(d\kappa)^\frac{1}{6}\log\left(\frac{1}{\epsilon}\right)^\frac{1}{6}},\min\left(\left(\frac{m}{d}\right)^\frac{1}{3},\left(\frac{Mm}{16\sigma^2}\right)^\frac{1}{3},\sqrt{m}\right)\epsilon^\frac{2}{3}\log\left(\frac{1}{\epsilon}\right)^{-\frac{2}{3}}\right)
	\quad  b=\frac{d\kappa}{h^{3}}\quad \nu=\frac{uh^2}{d^{1.5}}\numberthis\label{eq:paramchoicefinal}
	\end{align}
for the ZO-RMP method described in \eqref{eq:rmpupdatexnh}-\eqref{eq:rmpupdatevn1}, with $u=1/M$, we have $W_2(\varpi_N,\pi)\leq \epsilon$ after 
	\begin{align*}
	N=\tilde{O}\left(\max\left(\frac{d^\frac{1}{6}\kappa^\frac{7}{6}}{(\epsilon\sqrt{m})^\frac{1}{3}},\frac{\kappa\max\left(\left(\frac{d}{m}\right)^\frac{1}{3},\left(\frac{\sigma^2}{Mm}\right)^\frac{1}{3},\frac{1}{\sqrt{m}}\right)}{\epsilon^\frac{2}{3}}\right)\right)\numberthis\label{eq:Nbound}
	\end{align*}
iterations. Hence, the total-number of zeroth-order oracle calls are given by
	\begin{align*}
	2Nb=\tilde{O}\left(\max\left(\frac{d^\frac{5}{3}\kappa^\frac{8}{3}}{\epsilon^\frac{4}{3}},\frac{d\kappa^2\max\left(\left(\frac{d}{m}\right)^\frac{1}{3},\left(\frac{\sigma^2}{Mm}\right)^\frac{1}{3},\frac{1}{\sqrt{m}}\right)^4}{\epsilon^\frac{8}{3}}\right)\right). \numberthis\label{eq:oraclebound}
	\end{align*}	
\end{theorem}
\begin{remark}
The analysis of the randomized midpoint algorithm in~\cite{shen2019randomized}, for the first-order setting, requires access to exact minimizer $x^*$ as the initializer. We relax this requirement to the having a point $x_0$ satisfying $\expec{f(x_{0})-f(x^*)}=O(d)$, which is a milder requirement. It is well-known from the stochastic optimization literature, that under Assumption~\ref{as:zo}, and \ref{smooth_assum}, in the zeroth-order setting, using the zeroth-order version of stochastic gradient algorith, the oracle complexity of finding a point $x_0$ such that $\expec{f(x_{0})-f(\bar{x})}=O(d)$ where $\bar{x}$ is the minimizer of $f$, is $O(\kappa \log d)$~\cite{duchi2015optimal, nesterov2017random}. 
\end{remark}
\begin{remark}
	Note that even though the iteration complexity of ZO-RMP still matches with RMP (except for the dimension dependence which is unavoidable in the zeroth-order setting), and is better than KLMC for all values of $\epsilon$, the oracle complexity for ZO-RMP is not uniformly better than ZO-KLMC for all $\epsilon$. However, observe that when $h=C\frac{(\epsilon\sqrt{m})^\frac{1}{3}}{(d\kappa)^\frac{1}{6}\log\left(\frac{1}{\epsilon}\right)^\frac{1}{6}}$, i.e., when $\epsilon\geq \max\left(\sqrt{\frac{d}{M}},\frac{16\sigma^2}{M^\frac{3}{2}\sqrt{d}},\frac{1}{\sqrt{dmM}}\right)$ the oracle complexity of ZO-RMP is $\tilde O\left(\frac{d^\frac{5}{3}\kappa^\frac{8}{3}}{\epsilon^\frac{4}{3}}\right)$ which is indeed better compared to $\tilde O\left(\frac{d^2}{\epsilon^2}\right)$ for ZO-KLMC.
\end{remark}

We end this section by mentioning that developing lower bounds on the oracle complexity of sampling from strongly log-concave densities, in the stochastic zeroth-order setting that we consider is an interesting open problem. 

\section{Oracle Complexity Results under Log-Sobolev Inequality}

The algorithms and oracle complexity results in the previous sections were stated for smooth and strongly log-concave densities (i.e., under Assumption~\ref{smooth_assum}), which covers important classes of problems in sampling and Bayesian inference. However, the fundamental idea behind the non-asymptotic convergence results of the discretization based sampling algorithm are essentially based on the following facts: (i) the underlying continuous (underdamped or overdamped) Langevin diffusion converges to its equilibrium state (i.e., to the target distribution $\pi$ in this case) exponentially fast in various metrics, and (ii) consequently, the potential function is smooth enough that the error due to discretization is not extremely large. Roughly speaking, condition \textbf{A1} and \textbf{A2} in Assumption~\ref{smooth_assum} corresponds respectively to the above facts, respectively. However, it is well-know that the overdamped Langevin diffusion converges to its equilibrium under much weaker conditions that strong log-concavity; indeed as long as the target density satisfies functional inequalities like Poincare or Log-Sobolev inequalities, the overdamped Langevin diffusion converges to its equilibrium exponentially faster in various metrics; see, for example~\cite{bakry2013analysis}. Motivated by the above fact, recently~\cite{vempala2019rapid} demonstrated that the LMC algorithm also exhibits rapid convergence to the target density if it has access to the exact gradient evaluations of the potential function $f$. As a consequence, one could sample from densities that are not essentially strongly log-concave, there by extending the applicability of LMC algorithms for a wider class of Bayesian inference problems. In this section, we analyze the performance of stochastic zeroth-order discretization of overdamped Langevin diffusions when the target density satisfies log-Sobolev inequality. 

\begin{assumption}\label{as:lsi}
	A density $\pi$ is said to satisfy Log-Sobolev Inequality (LSI) with a constant $\lambda>0$ if for all smooth function $g:\mathbb{R}^d\rightarrow \mathbb{R}$ with finite variance,
	\begin{align*}
	\int_{\mathbb{R}^d} g^2(\theta) \log g^2 (\theta) \pi(\theta) d\theta-\int_{\mathbb{R}^d} g^2(\theta) \pi(\theta) d\theta \log \int_{\mathbb{R}^d}g^2 (\theta) \pi(\theta) d\theta \leq \frac{2}{\lambda} \int_{\mathbb{R}^d} \|\nabla g(\theta)\|^2 \pi(\theta) d\theta.\numberthis\label{eq:lsidef}
	\end{align*}
\end{assumption}
In Section~\ref{sec:lsinumerics}, we show that mixture of Gaussian densities with unequal covariance satisfies the above assumption, while it does not satisfy condition~\textbf{A1} of Assumption~\ref{smooth_assum}, and discuss applications to Bayesian variable selection.  The above assumption also leads to the following equivalent formulation. Let $H_{\pi}(\varpi)$, and $J_{\pi}(\varpi)$ be the Kullback-Leibler (KL) divergence of $\varpi$ with respect to $\pi$, and the relative Fisher Information respectively which are defined as follows:
\begin{align*}
H_{\pi}(\varpi)=\int_{\mathbb{R}^d}\varpi(\theta)\log\frac{\varpi(\theta)}{\pi(\theta)}d\theta \quad J_{\pi}(\varpi)=\int_{\mathbb{R}^d}\varpi(\theta)\left\|\nabla \log\frac{\varpi(\theta)}{\pi(\theta)}\right\|^2d\theta.\numberthis\label{eq:HJdef}
\end{align*}
One can verify that LSI is equivalent to the following condition by plugging $g^2=\varpi/\pi$ in \eqref{eq:lsidef}:
\begin{align*}
H_{\pi}(\varpi)\leq\frac{1}{2\lambda}J_{\pi}(\varpi).\numberthis\label{eq:JHrel}
\end{align*}
We also know that when $\pi$ satisfies LSI, Talagrand inequality holds~\cite{bakry2013analysis}, i.e., for all $\varpi$,
\begin{align*}
\frac{\lambda}{2}W_2(\varpi,\pi)^2\leq H_\pi(\varpi).\numberthis\label{eq:w2hinequality}
\end{align*}
With this background, we provide our oracle complexity result of ZO-LMC algorithm when the density satisfies LSI and is smooth. 
\begin{theorem}\label{th:lsilmc}
	Let the target density $\pi$ satisfy Assumption~\ref{as:lsi} and let the potential function $f$ be satisfy condition \textbf{A2} in Assumption~\ref{smooth_assum}. Let $x_0 \sim \varpi_0$ which satisfies $H_\pi(\varpi_0) \leq \infty$. Then for the ZO-LMC update as in \eqref{eq:zolmcupdate}, under Assumption~\ref{as:zo}, by choosing,
	\begin{align*}
	b=\frac{384M^2(d+5)\max(1,\sigma^2)}{h\lambda^2}, \quad \nu=\frac{\sqrt{h}}{d+3}, \quad h=\frac{\epsilon^2}{d},\numberthis\label{eq:lsilmcparamchoice}
	\end{align*}
	we have $W_2(\varpi_N,\pi)\leq \epsilon$, for all $0\leq \epsilon \leq \frac{\alpha}{4L^2}$, after $N$ iterations where
	\begin{align*}
	N=\tilde O\left(\frac{d}{ \epsilon^2}\right).\numberthis\label{eq:lsilmcNbound}
	\end{align*}
	Hence, the total number of calls to the stochastic zeroth-order oracle is given by
	\begin{align*}
	Nb=\tilde O\left(\frac{\max(1,\sigma^2)~d^3}{\epsilon^4}\right) \numberthis\label{eq:lsilmcoraclebound}
	\end{align*}
\end{theorem}

\begin{remark}
Note that in comparison to condition \textbf{A1} of Assumption~\ref{smooth_assum}, the assumptions required for the above theorem are weaker. Specifically, in place of condition \textbf{A1} in Assumption~\ref{smooth_assum}, we have Assumption~\ref{as:lsi}. Condition \textbf{A2} is regarding the smoothness is required to handle error that arises due to discretization of continuous time dynamics. For this wider class of densities, the price to pay is that the dependency on both the dimension $d$ and $\epsilon$ increases in comparison to Theorem~\ref{thm:zlmc}.
\end{remark}

\begin{remark}
Given that ZO-LMC exhibits convergence (albeit with a slightly weaker $\epsilon$ and $d$ dependency, it is natural to ask if ZO-KLMC also exhibits similar convergence. However, even in the first-order setting this question is open. Indeed, kinetic Langevin diffusions are 
 a class of degenerate diffusions which require a different class of function inequalities (called as hypocoercivity~\cite{dric2009hypocoercivity}) for them to converge to their equilibrium. It is an open question to show that the discretize sampling algorithm (KLMC or appropriate modifications) also convergence under hypocoercivity and appropriate smoothness assumptions on the potential function $f$, either given access to exact first-order oracles or stochastic zeroth-order oracles. We leave this question as future work.
  \end{remark}

%% file: zo_noise.tex
\section{One-Point Setting: Independent noise per function evaluation}\label{sec:noise}
As discussed in Section~\ref{sec:mainmethod}, there are subtle differences between the availability of one and two-point evaluation based stochastic zeroth-order gradients. In this section, we examine this difference in more detail. Recall that while defining the zeroth-order gradient estimator in \eqref{eq:gradest}, we assumed that the function can be evaluated at two points, namely, $\theta+\nu u_i$, and $\theta$, with the same noise $\xi_i$. This implies, when the noise is additive, i.e., $F(\theta,\xi)=f(\theta)+\xi$, the gradient estimator is not affected by the noise. Because in that case we have, $F(\theta+\nu u_i,\xi_i)-F(\theta,\xi_i)=f(\theta+\nu u_i)-f(\theta)$. We emphasize that this is our main reason for consider general non-additive noise in the previous sections. For example, under multiplicative noise, consider the case where $F(\theta,\xi)=\xi f(\theta)$, $\expec{\xi}=1$, and $f(\theta)$ is $L$-Lipschitz continuous; then Assumption~\ref{as:zo} holds. 

Now we will examine the one-point setting in that the noise in the two function evaluations of the gradient estimator is not the same. Specifically, first we show that allowing the noise $\xi_i$, and $\xi_i'$ in $F(\theta+\nu u_i,\xi_i)$, and $F(\theta,\xi_i')$ to be independent additive noise, deteriorates the iteration and/or oracle complexities of zeroth-order discretizations considered in the previous settings. Formally, we work under the following assumption in the one-point stochastic zeroth-order setting.
\begin{assumption}\label{as:additive}
	The stochastic zeroth-order oracle is such that for each point $x$, the observed function evaluation $F(\theta,\xi)$ is given by $F(\theta,\xi)=f(\theta)+\xi$ where $\expec{\xi}=0$, and $\expec{\xi^2}=\sigma^2$.
\end{assumption}
Under Assumption~\ref{as:additive}, the upper bound on the variance of the gradient estimator as stated in Lemma~\ref{lm:zogradvar} no longer holds. Instead, we have the following result.

\begin{lemma}\label{lm:zogradesterrordiffnoise}
Let $g_{\nu,b}(\theta)$ in~\eqref{eq:gradest}, be defined under the one-point setting. Then under Assumption~\ref{as:additive} and condition A1 of Assumption~\ref{smooth_assum}, we have
	\begin{align*}
	\expec{\left\|g_{\nu,b}(\theta)-\nabla f_\nu(\theta)\right\|^2}\leq & \frac{2(d+5)\|\nabla f(\theta)\|^2}{b}+\frac{\nu^2 M^2(d+3)^3}{2b}+\frac{2d\sigma^2}{b\nu^2},\\
	\expec{\left\|g_{\nu,b}(\theta)-\nabla f(\theta)\right\|^2}\leq & \frac{4(d+5)\left(\|\nabla f(\theta)\|^2+\frac{\sigma^2}{\nu^2}\right)}{b}+\frac{3\nu^2 M^2(d+3)^3}{2}+\frac{4d\sigma^2}{b\nu^2}.
	\end{align*}
\end{lemma}
The main difference in the one-point setting, in terms of the variance of the gradient estimator is the presence of the third term, which is of the order of $1/b\nu^2$. This causes the additional difficulties in terms of setting the parameters $b$ and $\nu$ in the zeroth-order gradient estimator, which in turn causes the oracle complexities to deteriorate.  Based on the above result on the variance, we provide the oracle complexity results for ZO-LMC, ZO-KLMC and ZO-RMP under Assumption~\ref{as:additive} on the stochastic zeroth-order oracle, in Theorem~\ref{thm:zlmcdiffnoise},~\ref{thm:KLMCthmsdiffnoise} and~\ref{th:rmpmainresultdiffnois} respectively. 
\begin{theorem}[ZO-LMC under Strong Log-concavity]\label{thm:zlmcdiffnoise}
	Let the potential function $f$ satisfy Assumption~\ref{smooth_assum}. Then, for ZO-LMC algorithms under Assumption~\ref{as:additive}, by choosing
	\begin{align*}
	h=\frac{\epsilon^2}{d^2}, \quad b=\frac{\max(1,\sigma^2)\cdot d}{\epsilon^2}, \quad \nu=\frac{\epsilon}{\sqrt{d}}, \numberthis\label{eq:lmcparamchoicediffnoise}
	\end{align*}
	we have $W_2(\varpi_N,\pi)\leq \epsilon$ for $0<\epsilon\leq \min\left(d\sqrt{\frac{2}{M+m}},\sqrt{\frac{m(d+5)}{8M^2}}\right)$, after $N$ iterations, where
	\begin{align*}
	N={O}\left(\frac{d}{\epsilon^2}\log\left(\frac{d}{\epsilon}\right)\right) .\numberthis\label{eq:LMCNbounddiffnoise}
	\end{align*}
	Hence, the total number of calls to the stochastic zeroth-order oracle is given by,
	\begin{align*}
	Nb={O}\left(\frac{\max(1,\sigma^2)~d^2}{\epsilon^4}\log\left(\frac{d}{\epsilon}\right)\right).\numberthis\label{eq:LMCoraclebounddiffnoise}
	\end{align*}
\end{theorem}
\begin{theorem}[ZO-KLMC under Strong Log-concavity]\label{thm:KLMCthmsdiffnoise}
	Let the function $f$ satisfy Assumption~\ref{smooth_assum}. If the initial point $(\tilde x_0,x_0)$ is chosen such that $\tilde x_0\sim N(0,\boldsymbol{I}_d)$, then, ensuring $\gamma\geq\sqrt{m+M}$, under Assumption~\ref{as:additive} for the ZO-KLMC, by choosing,
	\begin{align*}
	h=\frac{m\epsilon}{12\gamma M\sqrt{d}}\quad \nu=\frac{\epsilon}{\sqrt{d}}\quad b=\frac{d^{1.5}\max(1,\sigma^2)}{\epsilon^3} \numberthis\label{eq:klmcparchoicediffnoise}
	\end{align*}
	we have $W_2(\varpi_N,\pi)\leq \epsilon$ for $0<\epsilon\leq \frac{12M\gamma^2\sqrt{d}}{m^2}$, after
	\begin{align*}
	N=\tilde O\left(\frac{\sqrt{d}}{\epsilon}\right)\numberthis\label{eq:klmcNbounddiffnoise}
	\end{align*}
	iterations. Hence, the total number of oracle calls to the stochastic zeroth-order oracle is given by
	\begin{align*}
	Nb=O\left(\frac{\max(1,\sigma^2)~d^{2}}{\epsilon^4}\right).\numberthis\label{eq:klmcoraclebounddiffnoise}
	\end{align*}
\end{theorem}
\begin{theorem}[ZO-RMP under Strong Log-concavity]\label{th:rmpmainresultdiffnois}
Let the potential function satisfy Assumption~\ref{smooth_assum} and let $x^*$ be the minimizer of $f$, $x_0$ be such that $\expec{f(x_{0})-f(\bar{x})}=O(d)$, and $v_0=0$. Let the stochastic zeroth-order oracle satisfy Assumption~\ref{as:additive}. Then, for $0\leq \epsilon\leq 1$, by choosing,
	\begin{align*}
	&h= C\min\left(\frac{(\epsilon\sqrt{m})^\frac{1}{3}}{(d\kappa)^\frac{1}{6}\log\left(\frac{1}{\epsilon}\right)^\frac{1}{6}},\min\left(\left(\frac{m}{d}\right)^\frac{1}{3},\left(\frac{Mm}{16\sigma^2}\right)^\frac{1}{3},\sqrt{m}\right)\epsilon^\frac{2}{3}\log\left(\frac{1}{\epsilon}\right)^{-\frac{2}{3}}\right)
	\quad  b=\frac{d^4\kappa}{h^{7}}\quad \nu=\frac{uh^2}{d^{1.5}}\numberthis\label{eq:paramchoicefinaldiffnois}
	\end{align*}
for the ZO-RMP described in \eqref{eq:rmpupdatexnh}-\eqref{eq:rmpupdatevn1}, we have $W_2(\varpi_N,\pi)\leq \epsilon$ after
	\begin{align*}
	N=\tilde{O}\left(\max\left(\frac{d^\frac{1}{6}\kappa^\frac{7}{6}}{(\epsilon\sqrt{m})^\frac{1}{3}},\frac{\kappa\max\left(\left(\frac{d}{m}\right)^\frac{1}{3},\left(\frac{\sigma^2}{Mm}\right)^\frac{1}{3},\frac{1}{\sqrt{m}}\right)}{\epsilon^\frac{2}{3}}\right)\right) \numberthis\label{eq:Nbounddiffnois}
	\end{align*}
iterations. Hence, the total-number of zeroth-order oracle calls are given by
	\begin{align*}
	2Nb=\tilde{O}\left(\max\left(\frac{d^\frac{16}{3}\kappa^\frac{10}{3}}{\epsilon^\frac{8}{3}},\frac{d^4\kappa^2\max\left(\left(\frac{d}{m}\right)^\frac{1}{3},\left(\frac{\sigma^2}{Mm}\right)^\frac{1}{3},\frac{1}{\sqrt{m}}\right)^8}{\epsilon^\frac{16}{3}}\right)\right).\numberthis\label{eq:oraclebounddiffnois}
	\end{align*}	
\end{theorem}
\begin{remark}
As before, the oracle complexity of ZO-RMP in this setting is not uniformly better than that of ZO-KLMC. We do observe that when $h=C\frac{(\epsilon\sqrt{m})^\frac{1}{3}}{(d\kappa)^\frac{1}{6}\log\left(\frac{1}{\epsilon}\right)^\frac{1}{6}}$, i.e., when $\epsilon\geq \max\left(\sqrt{\frac{d}{M}},\frac{16\sigma^2}{M^\frac{3}{2}\sqrt{d}},\frac{1}{\sqrt{dmM}}\right)$ the oracle complexity of ZO-RMP is $\tilde O\left(\frac{d^\frac{16}{3}\kappa^\frac{10}{3}}{\epsilon^\frac{8}{3}}\right)$ which is worse compared to $\tilde O\left(\frac{d^2}{\epsilon^4}\right)$ for ZO-KLMC. However, it is better than that of ZO-KLMC in the opposite regime.
\end{remark}

We now present the corresponding result when the target density is not strongly log-concave but satisfies LSI.
\begin{theorem}[ZO-LMC under Log-Sobolev Inequality]\label{th:lsilmcdiffnoise}
	Let the target density $\pi$ satisfy Assumption~\ref{as:lsi} and let the potential function $f$ satisfy condition \textbf{A2} of Assumption~\ref{smooth_assum}. Let $x_0 \sim \varpi_0(x)$ which satisfies $H_\pi(\varpi_0) \leq \infty$. Then for the ZO-LMC update as in \eqref{eq:zolmcupdate}, under Assumption~\ref{as:additive}, by choosing,
	\begin{align*}
	b=\frac{384M^2(d+5)\max(1,\sigma^2)}{h^2\lambda^2},\quad \nu=\frac{\sqrt{h}}{d+3},\quad h=\frac{\epsilon^2}{d},\numberthis\label{eq:lsilmcparamchoicediffnoise}
	\end{align*}
	we have, $W_2(\varpi_N,\pi)\leq \epsilon$, for all $0\leq \epsilon \leq \frac{\lambda}{4L^2}$ after
	\begin{align*}
	N=\tilde O\left(\frac{d}{ \epsilon^2}\right) \numberthis\label{eq:lsilmcNbounddiffnoise}
	\end{align*}
	iterations. Hence, the total number of calls to the zeroth-order oracle is given by
	\begin{align*}
	Nb=\tilde O\left(\frac{d^4}{\epsilon^6}\right). \numberthis\label{eq:lsilmcoraclebounddiffnoise}
	\end{align*}
\end{theorem}

\begin{remark}[\textit{One-point setting with non-additive noise} ]Given the above result, it is natural to examine the effect of non-additive noise on the oracle complexities. For this case, we have the following result under an additional smoothness assumption on the stochastic function evaluations $F(x,\xi)$.
\begin{lemma}\label{lemma:lipfunc}
	Let the function $F(\theta,\xi)$ be Lipschitz continuous in its second argument, i.e., $|F(\theta,\xi)-F(\theta,\xi')|\leq L|\xi-\xi'|$. Under the above condition,  Lemma~\ref{lm:zogradesterrordiffnoise} holds. Consequently, all the above complexity results in this section holds. 
\end{lemma}
\end{remark}

\begin{remark}[\textit{Effect of Higher-order smoothness}]
While the oracle complexities under the one-point evaluation setting are worse than that of the two-point setting, they could be made to approach that of the two-point setting when we make the stronger assumption that the potential function is assumed to the $\beta$-times differentiable and the $(\beta-1)$-derivatives are Lipschitz continuous. Similar phenomenon has been observed in the case of highly-smooth convex stochastic zeroth-order optimization; see, for example~\cite{bach2016highly, akhavan2020exploiting}. As the precise statements and the proofs are similar to that of the above theorems, we omit the details.    
\end{remark}

%% file: hdzlmc.tex
\section{Variable Selection for High-dimensional Black-box Sampling}\label{sec:hdzlmc}
In practical black-box settings, due to the non-availability of the analytical form of $f(\theta)$, one might potentially over-parametrize $f(\theta)$, in terms of number of covariates selected for modeling. Hence, the problem of variable selection, in a zeroth-order setting becomes crucial. To address this issue, in this section, we study variable selection under certain sparsity assumptions on the objective function $f$, to facilitate sampling in high-dimensions. Throughout this section, we assume one could observe exact function evaluations, without any noise. We emphasize that we make this assumption purely for technical convenience and to convey the theoretical results insightfully; all results presented in this section extends to the noisy setting in a straightforward manner. Specifically, we make the following assumption on the structure of $f$.
\begin{assumption}\label{spars_assum}
We assume that $f(\theta): \mathbb{R}^d \to \mathbb{R}$ is $s$ sparse, i.e., the function $f$ depends only on (the same) $s$ of the $d$ coordinates, for all $\theta$, where $s \ll d$. We denote the true support set as $S^*$. This implies that for any $\theta \in \mathbb{R}^d$, we have $\| \nabla f(\theta) \|_0 \leq s$, i.e., the gradient is $s$-sparse. Furthermore, define $\nabla f_\nu(\theta) = \E_u \left[\nabla f(\theta+ \nu u) \right]$ for a standard gaussian random vector $u$. Then the gradient sparsity assumption also implies that  $\| \nabla f_\nu(\theta) \|_0 \leq s$ for all $\theta \in \mathbb{R}^d$.  Furthermore, we assume that the gradient lies in the following set that characterizes the minimal signal strength in the relevant coordinates of the gradient vector:
\begin{align*}
\mathcal{G}_{a,s} = \left\{ \nabla f(\theta)  : \| \nabla f(\theta) \|_0 \leq s~\text{and}~\sup_{\theta \in \mathbb{R}^d} \inf_{j \in S^*}| [\nabla f(\theta) ]_j |\geq a  \right\}
\end{align*}
\end{assumption}
As a consequence, we also have that $\nabla f_\nu(\theta) \in \mathcal{G}_{a,s}$. The above assumption makes a \emph{homogenous} sparsity assumption on the sparsity and the minimum signal strength of the gradient. Roughly speaking, $a$ represents the minimum signal strength in the gradient so that efficient estimation of the support $S^*$ is possible in the sample setting. The above sparsity model on the function $f$, converts the problem to variable selection in a non-Gaussian sequence model setting:
\begin{align*}
[g_{\nu,n} ]_j=   [\nabla f_\nu(\theta) ]_j  + \zeta_j \qquad j = 1,\ldots, d.
\end{align*}
Hence, $\zeta_j$ are zero-mean random variables as $[g_{\nu,n} ]_j$ is an unbiased estimator of $ [\nabla f_\nu(\theta) ]_j$. We refer the reader to~\cite{butucea2018variable} for recent results on variable selection consistency in Gaussian sequence model setting. We also make the following assumption on the query point selected to estimate the gradient.
\begin{assumption}\label{queryassmp}
The query point $\theta \in \mathbb{R}^d$ selected is such that $\| \nabla f(\theta) \|_2 \leq R$.
\end{assumption}
Our algorithm for high-dimensional black-box sampling with variable selection is as follows:
\begin{itemize}
\item Pick a point $\theta$ (which is assumed to satisfy Assumption~\ref{queryassmp}) and estimate the gradient $g_{\nu,n}$ at that point and compute the estimator $\hat S$ of $S^*$ as $\hat S = \{j: |[g_{\nu,n}]_j| \geq \tau\}$.
\item Run any of the zeroth-order sampling algorithm on the selected set of coordinates $\hat{S}$ of $f(\theta)$.
\end{itemize}
Here, for the first step, we need to select $n, \tau$ and $\nu$. We separate the set of relevant variables by thresholding $|[g_{\nu,n}]_j|$ at $\tau$. We now provide our result on the probability of erroneous selection.
\begin{theorem}\label{thm:hdzlmc}
Let $f$ satisfy Assumption~\ref{smooth_assum} and the query point selected satisfy Assumption~\ref{queryassmp}. Set $ \tau = (a-M\nu\sqrt{s})/2$ and assume that $\nu\leq \min\left(\frac{a}{2M\sqrt{s}}, \frac{R}{MC_2\sqrt{s}}\right)$ and $$n \geq \max \left( \frac{8RC\sqrt{s}}{a}\left(\frac1{K_2}\log\frac{4d}{\epsilon}\right)^{3\slash2},~~K_1\frac{8RC\sqrt{s}}{a},~~\left(\frac{8RC\sqrt{s}}{a}\right)^4\right)$$ where $C,C_2$ are constants. Then we have $\Pr\{\hat S\neq S^*\} \leq \epsilon$. 
\end{theorem}
\begin{remark}
The number of queries $n$ to the function $f$ depends only logarithmically on the dimension $d$ and is a (low-degree) polynomial in the sparsity level $s$. Combining this fact with the result in Theorem~\ref{thm:zlmc} we see that the total number of queries to the function $f$ (for the sampling error measured in 2-Wasserstein distance) is only poly-logarithmic in the true dimension $d$ and is a low-degree polynomial in the sparsity level $s$. Thus when $s \ll d$, we see the advantage of variable selection in black-box sampling using the two-step approach. The above results assumes that the sparsity level $s$ and signal strength is known. It would be interesting to construct adaptive estimators similar to those for Gaussian sequence model in~\cite{butucea2018variable}. Furthermore, exploring appropriately defined notions of non-homogenous sparsity assumptions is also challenging.
\end{remark}

%% file: appendix.tex
\begin{center}
\textbf{\large Stochastic Zeroth-order Discretizations of Langevin Diffusions for Bayesian Inference: Supplementary Material}
\end{center}
\section{Notations}\label{sec:notations}
We use $a \wedge b$ and $a \vee b$ to denote the minimum and maximum of $a$ and $b$ respectively. The $L_2$ norm of a random vector $X:\Omega\to\R^d$ is defined to be $\|X\|_{L_2} = \E[\|X\|_2^2]^{1\slash2}$. The $L_p$ norms of a random matrix $\boldsymbol{M}:\Omega\to\R^{d\times d}$ are defined as follows.
\begin{align*}
\|\boldsymbol{M}\|_{L_p,2} &= \E[\|\boldsymbol{M}\|_2^p]^{1\slash p}, \\
\|\boldsymbol{M}\|_{L_p,F} &= \E[\|\boldsymbol{M}\|_F^p]^{1\slash p},
\end{align*}
where $\|\cdot\|_2$ is the spectral norm, and $\|\cdot\|_F$ is the Frobenius norm. For simplicity, we write $\|\cdot\|=\|\cdot\|_2$ and $\|\cdot\|_{L_p}=\|\cdot\|_{L_p,\boldsymbol{\bullet}}$ when there is no ambiguity. Furthermore, we omit the subscript $h$ in $x_{t,h}$ in places where is no confusion for simplicity.

\section{Proofs for Section~\ref{sec:zolmc}}
\subsection{Proofs for Oracle Complexity of ZO-LMC}
\begin{proof}[of Theorem~\ref{thm:zlmc}]
The proof follows by first calculating the bias and variance of the gradient estimator in our zeroth-order setting, where the error term $\zeta_n = g_{\nu,b}(x_n)-\nabla f(x_n)$. First, by Stein's identity, $\E[g_{\nu,1}(x,u)] = \E[\nabla f(x+\nu u)]=\nabla f_\nu(x)$, where we denote $f_\nu(x) = \E[f(x+\nu u)]$. Under Assumption \ref{smooth_assum} on smoothness of $f$, in the case where $b=1$, we have the following calculation for the bias.
\begin{align*}
\|\E[\zeta_n\mid x_n]\|^2
= \|\E[\nabla f(x_n+\nu u)\mid x_n]-\nabla f(x_n)\|^2
\leq \E[(M\nu\|u\|)^2] \numberthis\label{eq:condexpeczeta}
\leq M^2\nu^2d.
\end{align*}

Next, for $b\geq1$ in general, $g_{\nu,b}(x) = \frac1b\sum_{k=1}^bg_{\nu,1}(x,u_k)$, the bias and variance could be calculated as follows. Specifically, for the bias, we have
\begin{align*}
\|\E[\zeta_n\mid x_n]\|^2
= \|\E[g_{\nu,b}(x_n)-\nabla f(x_n)\mid x_n]\|^2 
\leq \|\E[g_{\nu,1}(x_n)-\nabla f(x_n)\mid x_n]\|^2 
\leq M^2\nu^2d.
\end{align*}
From Lemma 2.1 of \cite{balasubramanian2018zeroth}, we have,
\begin{align*}
\E[\|\zeta_n-\E[\zeta_n\mid x_n]\|^2]
\leq \frac{\nu^2}{2b}M^2(d+3)^3+\frac{2(d+5)\left(\sigma^2+\|\nabla f(x_n)\|_{L_2}^2\right)}{b}.
\end{align*}

Next, we follow a similar framework to the proof of Theorem 4 in \cite{dalalyan2017user}, but with modifications to adapt to the variance that is not uniformly bounded. Recall that $\Delta_n = L_0 - x_n,\;\Delta_{t+1} = L_h - x_{t+1}$, where $L_n = L_0 - \int_0^T\nabla f(L_s)ds + \sqrt2W_n$ follows the Langevin diffusion with stationary distribution $\pi$. Moreover, $\|\Delta_n - hU\| = \|\Delta_n - h[\nabla f(x_n+\Delta_n) - \nabla f(x_n)]\|\leq(1-mh)\|\Delta_n\|,\; \|V\| = \|\int_0^h[\nabla f(L_s)-\nabla f(L_0)]ds\|\leq1.65M(h^3d)^{1\slash2}$. Thus,
\begin{align*}
\|\Delta_{n+1}\|_{L_2}
&= \|\Delta_n - hU - V + h\zeta_n\|_{L_2} \\
&\leq \{\|\Delta_n - hU\|_{L_2}^2 + h^2\|\zeta_n - \E[\zeta_n\mid x_n]\|_{L_2}^2\}^{1\slash2} + \|V\|_{L_2} + h\|\E[\zeta_n\mid x_n]\|_{L_2} \\
&\leq \left\{(1-mh)^2\|\Delta_n\|_{L_2}^2 + h^2\left(\frac{\nu^2}{2b}M^2(d+3)^3+\frac{2(d+5)\left(\sigma^2+\|\nabla f(x_n)\|_{L_2}^2\right)}{b}\right)\right\}^{1\slash2} \\
& + 1.65M(h^3d)^{1\slash2} + M\nu hd^{1\slash2} \\
&\leq \left\{(1-mh)^2\|\Delta_n\|_{L_2}^2 + h^2\left(\frac{\nu^2}{2b}M^2(d+3)^3+\frac{2(d+5)\left(\sigma^2+2M^2\|\Delta_n\|_{L_2}^2+2\|\nabla f(L_0)\|_{L_2}^2\right)}{b}\right)\right\}^{1\slash2} \\
&+ 1.65M(h^3d)^{1\slash2} + M\nu hd^{1\slash2} \\
&\leq \left\{(1-mh)^2\|\Delta_n\|_{L_2}^2 + h^2\left(\frac{\nu^2}{2b}M^2(d+3)^3+\frac{2(d+5)\left(\sigma^2+2Md\right)}{b}\right)\right\}^{1\slash2}\\
&+ \frac{4M^2h^2(d+5)}{b(1-mh)}\|\Delta_n\|_{L_2} 
 + 1.65M(h^3d)^{1\slash2} + M\nu hd^{1\slash2}\\
&\leq \left\{(1-mh)^2\|\Delta_n\|_{L_2}^2 + h^2\left(\frac{\nu^2}{2b}M^2(d+3)^3+\frac{2(d+5)\left(\sigma^2+2Md\right)}{b}\right)\right\}^{1\slash2}\\
&+ \frac{mh}{2}\|\Delta_n\|_{L_2} 
+ 1.65M(h^3d)^{1\slash2} + M\nu hd^{1\slash2}.
\end{align*}
Here we use the fact that $\sqrt{a^2+b+c}\leq\sqrt{a^2+b}+\frac{c}{2a}$, $\E[\|\nabla f(L)\|^2]\leq Md$, and that we choose $h$, and $b$ such that $h/\left(b(1-mh)\right)\leq m/(8M^2(d+5))$. By Lemma 9 in \cite{dalalyan2017user}, the above inequality leads to
\begin{align*}
\|\Delta_n\|_{L_2}
&\leq (1-0.5mh)^n\|\Delta_0\|_{L_2} + \frac{3.3M\sqrt{hd}}{m}+\frac{2\nu M\sqrt{d}}{m}+\frac{\nu M\sqrt{h}}{2\sqrt{mb}}(d+3)^\frac{3}{2}+\frac{3\sqrt{h(d+5)(\sigma^2+2Md)}}{\sqrt{mb}}.
\end{align*}
Therefore, using the fact $W_2(\varpi_{n+1},\pi)\leq \|\Delta_{n+1}\|_{L_2}$, and $W_2(\varpi_{0},\pi)= \|\Delta_{0}\|_{L_2}$, we obtain the bound in Wasserstein distance.
\begin{align*}
W_2(\varpi_n,\pi)
\leq & (1-0.5mh)^nW_2(\varpi_0,\pi)  + \frac{3.3M\sqrt{hd}}{m}+\frac{2\nu M\sqrt{d}}{m}+\frac{\nu M\sqrt{h}}{2\sqrt{mb}}(d+3)^\frac{3}{2}+\frac{3\sqrt{h(d+5)(\sigma^2+2Md)}}{\sqrt{mb}}.\numberthis\label{eq:lmcW2recfinal}
\end{align*}
Choosing $h$, $b$, $\nu$, and $N$ as in \eqref{eq:lmcparamchoice}, and \eqref{eq:LMCNbound} we have $W_2(\varpi_N,\pi)\leq \epsilon$. 
\end{proof}

\subsection{Proofs for Oracle Complexity of ZO-KMLC}

\begin{proof}[of Theorem~\ref{thm:KLMCthms}]
	Let $(V_{n,n}, L_{n,t}),\;t\in[0,h]$ be a stationary kinetic Langevin process for each $n\in\mathbb{N}$, i.e.,
	\begin{align*}
	dV_{n,t} &= -(\gamma V_{n,t} + \nabla f(L_{n,t})) dt + \sqrt{2\gamma}dW_{n,t}, \\
	dL_{n,t} &= V_{n,t} dt,
	\end{align*}
	starting from $V_{0,0}\sim N(0,\boldsymbol{I}_d),\; L_{0,0}\sim\pi$, and satisfying $V_{n,h}=V_{n+1,0},L_{n,h}=L_{n+1,h}$. Define $(\tilde V_{n,t},\tilde L_{n,t})$ by the following discretized version of kinetic Langevin diffusion,
	\begin{align*}
	d\tilde V_{n,t} &= -(\gamma\tilde V_{n,t} + g(\tilde L_{n,0})) dt + \sqrt{2\gamma}dW_{n,t}, \\
	d\tilde L_{n,t} &= \tilde V_{n,t} dt,
	\end{align*}
	or equivalently,
	\begin{align*}
	\tilde V_{n,t} &= e^{-\gamma t}\tilde V_{n,0} - \int_0^te^{-\gamma(t-s)}ds\cdot g(\tilde L_{n,0}) + \sqrt{2\gamma}\int_0^te^{-\gamma(t-s)}dW_{n,t}, \\
	\tilde L_{n,t} &= \tilde L_{n,0} + \int_0^t\tilde V_{n,s}ds. 
	\end{align*}
	Define a different kinetic Langevin process $(\hat V_{n,t},\hat L_{n,t})$ with initial condition $\hat V_{n,0}=\tilde V_{n,0},\hat L_{n,0}=\tilde L_{n,0}$, i.e.,
	\begin{align*}
	d\hat{V}_{n,t} &= -(\gamma \hat{V}_{n,t} + \nabla f(\hat{L}_{n,t})) dt + \sqrt{2\gamma}dW_{n,t}, \\
	d\hat{L}_{n,t} &= \hat{V}_{n,t} dt
	\end{align*}
	Assume that $(\tilde V_{0,0},\tilde L_{0,0})$ is chosen such that $\tilde V_{0,0}=V_{0,0}$ and $W_2(\varpi_0,\pi)=\|\tilde L_{0,0}-L_{0,0}\|_{L_2}$. By definition of Wasserstein distance, we have $W_2(\varpi_n,\pi)\leq\|\tilde L_{n,0}-L_{n,0}\|_{L_2}$.\\
	Now we denote $e_n = \left\|\boldsymbol{P}^{-1}\begin{bmatrix}\tilde V_{n,0}-V_{n,0}\\\tilde L_{n,0}-L_{n,0}\end{bmatrix}\right\|_{L_2}$, where $\boldsymbol{P}^{-1}=\begin{bmatrix}\boldsymbol{I}_d&\gamma\boldsymbol{I}_d\\-\boldsymbol{I}_d&\boldsymbol{0}\end{bmatrix},\; \boldsymbol{P}=\gamma^{-1}\begin{bmatrix}0&-\gamma\boldsymbol{I}_d\\\boldsymbol{I}_d&\boldsymbol{I}_d\end{bmatrix}$ corresponds to the contraction to the kinetic Langevin process. See \cite{dalalyan2018sampling}. Note that $\|\tilde L_{n,0}-L_{n,0}\|_{L_2}\leq\sqrt2\gamma^{-1}e_n$ and $\|\tilde V_{n,0}-V_{n,0}\|_{L_2}\leq e_n$. 
	Observe that,
	\begin{align*}
	\tilde{V}_{n,h}-\hat{V}_{n,h}=&\int_{0}^{h}e^{-\gamma (h-s)}\left(\nabla f(\hat{L}_{n,s})-g_{\nu,b}(\hat{L}_{n,0})\right)ds\\
	\leq & \int_{0}^{h}e^{-\gamma (h-s)}\left(\nabla f(\hat{L}_{n,s})-\nabla f(\hat{L}_{n,0})+\nabla f(\hat{L}_{n,0})-g_{\nu,b}(\hat{L}_{n,0})\right)ds\\
	\leq &\int_{0}^{h}e^{-\gamma (h-s)}\left(\nabla f(\hat{L}_{n,s})-\nabla f(\hat{L}_{n,0})-\hat{\zeta}_{n,0}+\expec{\hat{\zeta}_{n,0}|\hat{L}_{n,0}}-\expec{\hat{\zeta}_{n,0}|\hat{L}_{n,0}}\right)ds\\
	\leq &\underbrace{\int_{0}^{h}e^{-\gamma (h-s)}\left(\nabla f(\hat{L}_{n,s})-\nabla f(\hat{L}_{n,0})\right)ds}_{A_1}-\underbrace{\int_{0}^{h}e^{-\gamma (h-s)}\left(\hat{\zeta}_{n,0}-\expec{\hat{\zeta}_{n,0}|\hat{L}_{n,0}}\right)ds}_{A_2}\\
	-&\underbrace{\int_{0}^{h}e^{-\gamma (h-s)}\left(\expec{\hat{\zeta}_{n,0}|\hat{L}_{n,0}}\right)ds}_{A_3}\numberthis\label{eq:klmcvtilvhatdiffexpansion}
	\end{align*}
	Similarly,
	\begin{align*}
	\tilde{L}_{n,h}-\hat{L}_{n,h}\leq &\underbrace{\int_{0}^{h}\int_{0}^{s}e^{-\gamma (s-u)}\left(\nabla f(\hat{L}_{n,u})-\nabla f(\hat{L}_{n,0})\right)duds}_{B_1}-\underbrace{\int_{0}^{h}\int_{0}^{s}e^{-\gamma (s-u)}\left(\hat{\zeta}_{n,0}-\expec{\hat{\zeta}_{n,0}|\hat{L}_{n,0}}\right)duds}_{B_2}\\
	-&\underbrace{\int_{0}^{h}\int_{0}^{s}e^{-\gamma (s-u)}\left(\expec{\hat{\zeta}_{n,0}|\hat{L}_{n,0}}\right)duds}_{B_3}\numberthis\label{eq:klmcltillhatdiffexpansion}
	\end{align*}
	Combining \eqref{eq:klmcvtilvhatdiffexpansion}, and \eqref{eq:klmcltillhatdiffexpansion}, we have
	\begin{align*}
	e_{n+1}
	=& \left\|\boldsymbol{P}^{-1}\begin{bmatrix}\tilde V_{n,h}-V_{n,h}\\\tilde L_{n,h}-L_{n,h}\end{bmatrix}\right\|_{L_2} \\
	\leq & \left\|\boldsymbol{P}^{-1}\begin{bmatrix} A_1-A_2-A_3\\ B_1-B_2-B_3\end{bmatrix}+\boldsymbol{P}^{-1}\begin{bmatrix}\hat V_{n,h}-V_{n,h}\\\hat L_{n,h}-L_{n,h}\end{bmatrix}\right\|_{L_2} \\
	\leq &\left\|\boldsymbol{P}^{-1}\begin{bmatrix}\hat V_{n,h}-V_{n,h}\\\hat L_{n,h}-L_{n,h}\end{bmatrix}-\boldsymbol{P}^{-1}\begin{bmatrix} A_2\\B_2\end{bmatrix}\right\|_{L_2}+\left\|\boldsymbol{P}^{-1}\begin{bmatrix} A_1\\B_1\end{bmatrix}\right\|_{L_2}+\left\|\boldsymbol{P}^{-1}\begin{bmatrix}  A_3\\B_3\end{bmatrix}\right\|_{L_2} \numberthis\label{eq:en1intermediateupperbound}
	\end{align*}
	Now we will upper bound the above three terms. 
	Observe that,
	\begin{align*}
	&\|A_1\|_{L_2}=\left\|\int_0^h e^{-\gamma(h-s)}(\nabla f(\hat L_{n,s})-\nabla f(\hat L_{n,0}))ds\right\|_{L_2} 
	\leq M\int_0^h \|\hat L_{n,s}-\hat L_{n,0}\|_{L_2} ds \\
	\leq& M\int_0^h\int_0^s\|\hat V_{n,u}\|_{L_2}duds 
	\leq \frac12Mh^2\max_{u\in[0,h]}\|\hat V_{n,u}\|_{L_2}. \numberthis\label{eq:A1sizebound}
	\end{align*}
	\begin{align*}
	\|B_1\|_{L_2}=\left\|\int_0^h \int_0^s e^{-\gamma(s-u)}(\nabla f(\hat L_{n,u})-\nabla f(\hat L_{n,0}))ds\right\|_{L_2}
	\leq \frac16Mh^3\max_{u\in[0,h]}\|\hat V_{n,u}\|_{L_2}. \numberthis\label{eq:B1sizebound}
	\end{align*}
	So, combining \eqref{eq:A1sizebound}, and \eqref{eq:B1sizebound}, and using the fact $\|\hat V_{n,u}\|_{L_2} \leq \|V_{n,u}\|_{L_2}+\|\hat V_{n,u}-V_{n,u}\|_{L_2} \leq \sqrt{d}+e_n$, and choosing $h\leq \sqrt{2}/(10\gamma)$, we obtain
	\begin{align*}
	\left\|\boldsymbol{P}^{-1}\begin{bmatrix} A_1\\B_1\end{bmatrix}\right\|_{L_2}\leq \sqrt{3}\|A_1\|_{L_2}+\sqrt{2}\gamma \|B_1\|_{L_2}\leq \frac{1}{2}Mh^2\left(\sqrt{3}+\frac{\sqrt{2}\gamma h}{3}\right)(\sqrt{d}+e_n)\leq Mh^3(\sqrt{d}+e_n).\numberthis\label{eq:pinva1b1bound}
	\end{align*}
	Using \eqref{eq:condexpeczeta} we have
	\begin{align*}
	\left\|\boldsymbol{P}^{-1}\begin{bmatrix} A_3\\B_3\end{bmatrix}\right\|_{L_2}\leq \sqrt{3}\|A_3\|_{L_2}+\sqrt{2}\gamma \|B_3\|_{L_2}\leq \left(\sqrt{3}h+\frac{\sqrt{2}\gamma h^2}{2}\right)M\nu \sqrt{d}\leq 2Mh\nu\sqrt{d}, \numberthis\label{eq:pinva3b3bound}
	\end{align*}
	and
	\begin{align*}
	&\left\|\boldsymbol{P}^{-1}\begin{bmatrix}\hat V_{n,h}-V_{n,h}\\\hat L_{n,h}-L_{n,h}\end{bmatrix}-\boldsymbol{P}^{-1}\begin{bmatrix} A_2\\B_2\end{bmatrix}\right\|_{L_2}^2\\
	=& \left\|\boldsymbol{P}^{-1}\begin{bmatrix}\hat V_{n,h}-V_{n,h}\\\hat L_{n,h}-L_{n,h}\end{bmatrix}\right\|_{L_2}^2+\left\|\boldsymbol{P}^{-1}\begin{bmatrix} A_2\\B_2\end{bmatrix}\right\|_{L_2}^2
	-2\expec{\begin{bmatrix}\hat V_{n,h}-V_{n,h}\\\hat L_{n,h}-L_{n,h}\end{bmatrix}^\top\begin{bmatrix}2\boldsymbol{I}_d&\gamma \boldsymbol{I}_d\\\gamma \boldsymbol{I}_d & \gamma^2\boldsymbol{I}_d\end{bmatrix}\begin{bmatrix} A_2\\B_2\end{bmatrix}}.
	\end{align*}
	Note that,
	\begin{align*}
	\left\|\boldsymbol{P}^{-1}\begin{bmatrix}\hat V_{n,t}-V_{n,t}\\\hat L_{n,t}-L_{n,t}\end{bmatrix}\right\|_{L_2} 
	\leq e^{-mt\slash\gamma} \left\|\boldsymbol{P}^{-1}\begin{bmatrix}\hat V_{n,0}-V_{n,0}\\\hat L_{n,0}-L_{n,0}\end{bmatrix}\right\|_{L_2}= e^{-mt\slash\gamma}e_n.\numberthis\label{eq:contklmccontrac}
	\end{align*}
	Using Lemma~\ref{lm:zogradvar}, we also have, 
	\begin{align*}
	\|A_2\|^2_{L_2}= &\left\|\int_{0}^{h}e^{-\gamma (h-s)}\left(\hat{\zeta}_{n,0}-\expec{\hat{\zeta}_{n,0}|\hat{L}_{n,0}}\right)ds\right\|_{L_2}^2=\frac{\left(1-e^{-\gamma h}\right)^2}{\gamma^2} \left\|\hat{\zeta}_{n,0}-\expec{\hat{\zeta}_{n,0}|\hat{L}_{n,0}}\right\|_{L_2}^2\\
	\leq& h^2\left\|g_{\nu,b}(\hat L_{n,0})-\nabla f_\nu(\hat L_{n,0})\right\|_{L_2}^2\leq \frac{2h^2(d+5)(\|\nabla f(\hat L_{n,0})\|_{L_2}^2+\sigma^2)}{b}+\frac{h^2\nu^2 M^2(d+3)^3}{2b}\\
	\leq & \frac{2h^2(d+5)(2\|\nabla f(\hat L_{n,0})-\nabla f( L_{n,0})\|_{L_2}^2+2\|\nabla f( L_{n,0})\|_{L_2}^2+\sigma^2)}{b}+\frac{h^2\nu^2 M^2(d+3)^3}{2b}\\
	\leq & \frac{2h^2(d+5)(2M^2\|\hat L_{n,0}- L_{n,0}\|_{L_2}^2+2\|\nabla f( L_{n,0})\|_{L_2}^2+\sigma^2)}{b}+\frac{h^2\nu^2 M^2(d+3)^3}{2b}.
	\end{align*}
	Using the fact that $\|\hat L_{n,0}-L_{n,0}\|_{L_2}^2\leq2\gamma^{-2}e_n^2$, and $\|\nabla f( L_{n,0})\|_{L_2}^2\leq Md$, we hence obtain
	\begin{align*}
	\|A_2\|^2_{L_2}\leq \frac{8M^2h^2(d+5)}{b\gamma^2}e_n^2+ h^2A_4
	\numberthis\label{eq:A2sizebound}
	\end{align*}
	where $A_4=\frac{2(d+5)(2Md+\sigma^2)}{b}+\frac{\nu^2 M^2(d+3)^3}{2b}$. Similarly, we have
	\begin{align*}
	\|B_2\|^2_{L_2}\leq \frac{2M^2h^4(d+5)}{b\gamma^2}e_n^2+ \frac{h^4}{4}A_4\numberthis\label{eq:B2sizebound}
	\end{align*}
	So, using \eqref{eq:A2sizebound}, and \eqref{eq:B2sizebound}, we have
	\begin{align*}
	\left\|\boldsymbol{P}^{-1}\begin{bmatrix} A_2\\B_2\end{bmatrix}\right\|_{L_2}^2
	\leq 3\|A_2\|_{L_2}^2 +2\gamma^2\|B_2\|_{L_2}^2 \leq \left(3h^2+\frac{\gamma^2h^4}{2}\right)\left(\frac{8M^2(d+5)}{b\gamma^2}e_n^2+A_4\right).\numberthis\label{eq:pinva2b2bound}
	\end{align*}
	Now using \eqref{eq:contklmccontrac}, \eqref{eq:pinva2b2bound}, and using the facts that, $\expec{A_2|\hat L_{n,0}}=0$, $\expec{B_2|\hat L_{n,0}}=0$, $\hat V_{n,h}-V_{n,h}$ is independent of $A_2,B_2$ given $\hat L_{n,0}$, and $\hat L_{n,h}-L_{n,h}$ is independent of $A_2,B_2$ given $\hat L_{n,0}$, we get
	\begin{align*}
	&\left\|\boldsymbol{P}^{-1}\begin{bmatrix}\hat V_{n,h}-V_{n,h}\\\hat L_{n,h}-L_{n,h}\end{bmatrix}-\boldsymbol{P}^{-1}\begin{bmatrix} A_2\\B_2\end{bmatrix}\right\|_{L_2}\\
	\leq &\left[ \left(8M^2\left(3h^2+\frac{\gamma^2h^4}{2}\right)\frac{d+5}{b\gamma^2}+e^{-\frac{2mh}{\gamma}}\right)e_n^2+4h^2A_4\right]^\frac{1}{2}\\
	\leq & \left[ \left(\frac{32M^2h^2(d+5)}{b\gamma^2}+e^{-\frac{2mh}{\gamma}}\right)e_n^2+4h^2A_4\right]^\frac{1}{2}\\
	\leq & \left[ \left(\frac{32M^2h^2(d+5)}{b\gamma^2}+\left(1-\frac{mh}{2\gamma}\right)^2\right)e_n^2+4h^2A_4\right]^\frac{1}{2}\\
	\leq & \left[ \left(1-\frac{mh}{4\gamma}\right)^2e_n^2+4h^2A_4\right]^\frac{1}{2}.\numberthis\label{eq:recursionfirsttermbound}
	\end{align*}
	The second inequality follows as $h\leq \sqrt{2}/(10\gamma)$, the third inequality follows if we choose $h\leq \min(\gamma/m,\sqrt{2}/(10\gamma))$, and the last inequality follows if we choose $b\geq\frac{512M^2(d+5)}{3m^2}$.
	Combining, \eqref{eq:en1intermediateupperbound}, \eqref{eq:pinva1b1bound}, \eqref{eq:pinva3b3bound}, and \eqref{eq:recursionfirsttermbound}, we get
	\begin{align*}
	e_{n+1}\leq \left[ \left(1-\frac{mh}{4\gamma}\right)^2e_n^2+4h^2A_4\right]^\frac{1}{2}+Mh^3e_n+Mh^3\sqrt{d}+2Mh\nu\sqrt{d}.
	\end{align*}
	Using Lemma 9 of \cite{dalalyan2017user}, and choosing $h\leq \min(\gamma/m,m/(12\gamma M))$ we have $mh/(4\gamma)-3Mh^2/2\geq mh/(8\gamma)$, and thus
	\begin{align*}
	&e_{n+1}\leq \left(1-\frac{mh}{8\gamma}\right)^{n+1}e_0+\frac{12M\gamma h\sqrt{d}}{m}+\frac{16M\nu \gamma \sqrt{d}}{m}+\frac{2h\sqrt{A_4}}{\sqrt{\frac{mh}{8\gamma}\left(2-\frac{mh}{4\gamma}-\frac{3Mh^2}{2}\right)}}\\
	\leq & \left(1-\frac{mh}{8\gamma}\right)^{n+1}e_0+\frac{12M\gamma h\sqrt{d}}{m}+\frac{16M\nu \gamma \sqrt{d}}{m}
	+\frac{4\sqrt{h}}{\sqrt{m}}\left(\frac{\sqrt{2(d+5)}(\sqrt{2Md}+\sigma)}{\sqrt{b}}+\frac{\nu M(d+3)^\frac{3}{2}}{\sqrt{2b}}\right)
	\end{align*}
	Then we obtain 
	\begin{align*}
	&W_2(\varpi_n,\pi)\\
	&\leq \sqrt2\gamma^{-1}e_n \\
	&\leq \sqrt2\gamma^{-1}\left(1-\frac{mh}{8\gamma}\right)^{n+1}W_2(\varpi_0,\pi)+\frac{24Mh\sqrt{d}}{m}+\frac{32M\nu  \sqrt{d}}{m}\\
	&+\frac{4\sqrt{h}}{\gamma\sqrt{m}}\left(\frac{2\sqrt{(d+5)}(\sqrt{2Md}+\sigma)}{\sqrt{b}}+\frac{\nu M(d+3)^\frac{3}{2}}{\sqrt{b}}\right) 
	\end{align*}
	Now, choosing $h$, $\nu$, $b$, and $N$ as in \eqref{eq:klmcparchoice}, we get \eqref{eq:klmcNbound}, and \eqref{eq:klmcoraclebound}.
\end{proof}

\subsection{Proofs for Oracle Complexity of ZO-RMP}

Before proceeding, we also recall that $\left(x_n^*(t),v_n^*(t)\right)$ when $t\in[0,h]$ is the true solution to the underdamped Langevin diffusion with the initial point $(x_n,v_n)$ coupled with $x_\nh$ through a shared Brownian motion defined as follows: 
\begin{align}
&x_n^*(t)=x_n+\frac{1-e^{-2t}}{2}v_n-\frac{u}{2}\int_0^t\left(1-e^{-2(t-s)}\right)\nabla f(x_n^*(s))ds+\sqrt{u}\int_0^t\left(1-e^{-2(t-s)}\right)dB_s \label{eq:xnstartdef}\\
&v_n^*(t)=v_ne^{-2t}-u\left(\int_0^te^{-2(t-s)}\nabla f(x_n^*(s))ds\right)+2\sqrt{u}\int_0^te^{-2(t-s)}dB_s
\label{eq:vnstartdef}.
\end{align}
We also recall some preliminary results from~\cite{shen2019randomized}.

\begin{lemma}[Lemma 6\cite{shen2019randomized}]\label{lm:contdiffprops}
	Let $\{x(t)\}_{t\in[0,h]}$, and $\{v(t)\}_{t\in[0,h]}$ be the true solution to the underdamped Langevin diffusion \eqref{eq:xnstartdef}, and \eqref{eq:vnstartdef} on ${t\in[0,h]}$. Then for $h\leq 1/20$, and $u=1/M$, we have
	\begin{align}
	&\expec{\sup_{t\in [0,h]}\|x(0)-x(t)\|^2}\leq O\left(h^2\|v(0)\|^2+u^2h^4\|\nabla f(x(0))\|^2+udh^3\right)\label{eq:supx0xtdiff}\\
	&\expec{\sup_{t\in[0,h]}\|\nabla f(x_t)\|^2}\leq O(\|\nabla f(x(0))\|^2+M^2h^2\|v(0)\|^2+Mdh^3) \label{eq:supgradsize}\\
	&\expec{\sup_{t\in[0,h]}\|v(t)\|^2}\leq O(\|v(0)\|^2+u^2h^2\|\nabla f(x(0))\|^2+udh)\label{eq:supvsize}
	\end{align}
\end{lemma}
\begin{lemma}\label{lm:graddiffxnhalf}
	Let $\alpha_n$ be sampled uniformly randomly from $[0,1]$ at iteration $n$. Let $x_\nh$ be the intermediate value at step $n$. Let $\{x_n^*(t)\}_{t\in[0,h]}$ be the true solution to  \eqref{eq:xnstartdef}, and \eqref{eq:vnstartdef} with the initial point $x_n^*(0)=x_n$ coupled to $x_\nh$ through a  shared Brownian motion. Then, under Assumption~\ref{as:zo} and Assumption~\ref{smooth_assum}, for $h\leq 1/20$, we have
	\begin{align*}
	\expec{\|\nabla f(x_\nh)-\nabla f(x_{n}^*(\alpha h))\|^2}
	\leq& O\bigg(M^2h^6\expec{\|v_n\|^2}+(h^8+h^{7}\kappa^{-1})\expec{\|\nabla f(x_n)\|^2}\\
	+&Mdh^7+h^{7}\kappa^{-1}\sigma^2+h^{8}\bigg).\numberthis\label{eq:graddiffxnhalf}
	\end{align*}
\end{lemma}
\begin{proof}[of Lemma~\ref{lm:graddiffxnhalf}] For notational simplicity, we drop the subscript $n$ from $\alpha_n$ below.
First, note that we have	
	\begin{align*}
	&\expec{\|\nabla f(x_\nh)-\nabla f(x_{n}^*(\alpha h))\|^2}\\
	\leq & M^2 \expec{\|x_\nh-x_{n}^*(\alpha h)\|^2}\\
	\leq & M^2 \expec{\|\frac{u}{2}\int_{0}^{\alpha h}(1-e^{-2(\alpha h - s)})(g_{\nu,b}(x_n)-\nabla f(x_{n}^*(s))ds\|^2}\\
	\leq & \frac{u^2M^2}{4} \expec{\int_{0}^{\alpha h}(1-e^{-2(\alpha h - s)})^2ds\int_{0}^{\alpha h}\|g_{\nu,b}(x_n)-\nabla f(x_{n}^*(s))\|^2ds}\\
	\leq & h^3 \expec{\int_{0}^{\alpha h}\|g_{\nu,b}(x_n)-\nabla f(x_{n}^*(s))\|^2ds}\\
	\leq & 2h^3 \expec{\int_{0}^{\alpha h}\left(\|g_{\nu,b}(x_n)-\nabla f(x_n)\|^2+\|\nabla f(x_n)-\nabla f(x_{n}^*(s))\|^2\right)ds}\\
	\leq & 2h^3 \expec{\int_{0}^{\alpha h}\left(\|g_{\nu,b}(x_n)-\nabla f(x_n)\|^2+M^2\|x_n-x_{n}^*(s)\|^2\right)ds}\\
	\leq & 2h^3 \textbf{E} \bigg[ \int_{0}^{\alpha h}\bigg(\frac{3\nu^2}{2}M^2(d+3)^3+\frac{4(d+5)\left(\sigma^2+\|\nabla f(x_n)\|^2\right)}{b}\\
	&~~~~+M^2O\left(h^2\|v_n\|^2+u^2h^4\|\nabla f(x_n)\|^2+udh^3\right)\bigg)ds \bigg]\\
	= & 2h^4 \expec{\frac{3\nu^2}{2}M^2(d+3)^3+\frac{4(d+5)\left(\sigma^2+\|\nabla f(x_n)\|^2\right)}{b}+M^2O\left(h^2\|v_n\|^2+u^2h^4\|\nabla f(x_n)\|^2+udh^3\right)}
	\end{align*}
	The first and the sixth inequality follows from the first condition of Assumption~\ref{smooth_assum}, the second inequality follows from \eqref{eq:rmpupdatexnh}, and \eqref{eq:xnstartdef}, the third inequality follows from Cauchy-Schwarz inequality, the fourth inequality follows from choosing $u=1/M$ and the fact that $1-e^{-2(\alpha h-s)}\leq 2h$, the fifth inequality follows from Young's inequality, the seventh inequality follows from Lemma~\ref{lm:zogradvar} and Lemma~\ref{lm:contdiffprops}.
	Choosing $b$, and $\nu$ as in \eqref{eq:paramchoicefinal}, we have, 
	\begin{align*}
	&\expec{\|\nabla f(x_\nh)-\nabla f(x_{n}^*(\alpha h))\|^2}\\
	\leq & O\left(M^2h^6\expec{\|v_n\|^2}+(h^8+h^{7}\kappa^{-1})\expec{\|\nabla f(x_n)\|^2}+Mdh^7+h^{7}\kappa^{-1}\sigma^2+h^{8}\right)
	\end{align*}
\end{proof}

\begin{lemma}\label{lm:graderrorxnhalf}
	Let $g_{\nu,b}(x_n)$ be defined as in \eqref{eq:gradest}. Then under the conditions of Lemma~\ref{lm:graddiffxnhalf}, we have
	\begin{align*}
	&\expec{\|\nabla f(x_\nh)-g_{\nu,b}(x_\nh)\|^2}\\
	\leq& O\left(M^2h^{5}\kappa^{-1}\ev2+h^{3}\kappa^{-1}\ef2+h^{4}+ Mdh^{6}\kappa^{-1}+h^{3}\kappa^{-1}\sigma^2\right) \numberthis\label{eq:graderrorxnhalf}
	\end{align*}
\end{lemma}
\begin{proof}[of Lemma~\ref{lm:graderrorxnhalf}]
	Using Lemma~\ref{lm:zogradvar} and Young's inequality, we have
	\begin{align*}
	&\expec{\|\nabla f(x_\nh)-g_{\nu,b}(x_\nh)\|^2}\leq \frac{3\nu^2}{2}M^2(d+3)^3+\frac{4(d+5)\left(\sigma^2+\expec{\|\nabla f(x_\nh)\|^2}\right)}{b}\\
	\leq & \frac{3\nu^2}{2}M^2(d+3)^3+\frac{4(d+5)\left(\sigma^2+2\expec{\|\nabla f(x_\nh)-\nabla f(x^*_n(\alpha h))\|^2}+2\expec{\|\nabla f(x^*_n(\alpha h))\|^2}\right)}{b}.\\
	\end{align*}
	Furthermore, using Lemma~\ref{lm:graddiffxnhalf}, and \eqref{eq:supgradsize}, and the fact that $h$ is small, we get
	\begin{align*}
	&\expec{\|\nabla f(x_\nh)-\nabla f(x^*_n(\alpha h))\|^2}+\expec{\|\nabla f(x^*_n(\alpha h))\|^2}\\
	\leq & O\left(M^2h^2\expec{\|v_n\|^2}+\expec{\|\nabla f(x_n)\|^2}+Mdh^{3}+h^{7}\kappa^{-1}\sigma^2+h^{8}\right).
	\end{align*}
	Hence, we have
	\begin{align*}
	&\expec{\|\nabla f(x_\nh)-g_{\nu,b}(x_\nh)\|^2}\\
	\leq& O\left(M^2h^{5}\kappa^{-1}\ev2+h^{3}\kappa^{-1}\ef2+h^{4}+ Mdh^{6}\kappa^{-1}+h^{3}\kappa^{-1}\sigma^2\right)
	\end{align*}
\end{proof}
\begin{lemma}\label{lm:lemma2eq}
	Let $\Ea$ denote the expectation with respect to $\alpha$ at each iteration $n$. Let $\expec{\cdot}$ be the expectation with respect to other randomness present in iteration $n$. Let $\{x_n^*(t)\}_{t\in[0,h]}$ be the true solution to  \eqref{eq:xnstartdef}, and \eqref{eq:vnstartdef} with the initial point $x_n^*(0)=x_n$ coupled to $x_\nh$, $v_n$, and $x_{n+1}$ through a shared Brownian motion. Then, under Assumption~\ref{as:zo}--\ref{smooth_assum}, for $h\leq 1/20$, and $u=1/M$, we have
	\begin{subequations}
		\begin{align*}
		\expec{\|\Ea{x_{n+1}}-x_n^*(h)\|^2}\leq& O\left((h^{10}+h^{9}\kappa^{-1})\expec{\|v_n\|^2}+u^2(h^{12}+h^{7}\kappa^{-1})\expec{\|\nabla f(x_n)\|^2}\right.\\
		&\left.+ud(h^{11}+h^{10}\kappa^{-1})+u^2h^{7}\kappa^{-1}\sigma^2+u^2h^{8}\right)\numberthis\label{eq:expecelphaxn1xnh}\\
		\expec{\|\Ea{v_{n+1}}-v_n^*(h)\|^2}\leq& O\left((h^{7}\kappa^{-1}+h^{8})\ev2+u^2(h^{10}+h^{5}\kappa^{-1})\ef2\right.\\
		&\left.+u^2h^{6}+u^2h^{5}\kappa^{-1}\sigma^2+ud(h^{9}+h^{8}\kappa^{-1})\right)\numberthis\label{eq:expecelphavn1vnh}\\
		\expec{\|x_{n+1}-x_n^*(h)\|^2}\leq &O\left(h^{6}\ev2+u^2h^{4}\ef2+u^2h^{8}+ udh^7+u^2h^{7}\kappa^{-1}\sigma^2\right) \numberthis\label{eq:xn1xnhdiff}\\
		\expec{\|v_{n+1}-v_n^*(h)\|^2}\leq &O\left(h^4\ev2+u^2h^4\ef2+u^2h^{8}+ udh^5+u^2h^{7}\kappa^{-1}\sigma^2\right)  \numberthis\label{eq:vn1vnhnormdiff}
		\end{align*}
	\end{subequations}
\end{lemma}
\begin{proof}[of Lemma~\ref{lm:lemma2eq}]
	\begin{enumerate}[label=\alph*)]
		\item Using Lemma~\ref{lm:graddiffxnhalf}, and \ref{lm:graderrorxnhalf}, we have
		\begin{align*}
		&\expec{\|\Ea{x_{n+1}}-x_n^*(h)\|^2}\\
		\leq& \expec{\|\frac{uh}{2}\Ea{(1-e^{-2(h-\alpha h)})g_{\nu,b}(x_\nh)}-\frac{u}{2}\int_{0}^{h}(1-e^{-2(h-s)})\nabla f(x_n^*(s))ds\|^2}\\
		\leq &  2\expec{\|\frac{uh}{2}\Ea{(1-e^{-2(h-\alpha h)})(g_{\nu,b}(x_\nh)-\nabla f(x_\nh))}\|^2}\\
		+&2\expec{\|\frac{uh}{2}\Ea{(1-e^{-2(h-\alpha h)})\nabla f(x_\nh)}-\frac{u}{2}\int_{0}^{h}(1-e^{-2(h-s)})\nabla f(x_n^*(s))ds\|^2}\\
		\leq &  2u^2h^4\expec{\|\Ea{(g_{\nu,b}(x_\nh)-\nabla f(x_\nh))}\|^2}\\
		+&2\expec{\|\frac{uh}{2}\Ea{(1-e^{-2(h-\alpha h)})\nabla f(x_\nh)}-\frac{u}{2}\int_{0}^{h}(1-e^{-2(h-s)})\nabla f(x_n^*(s))ds\|^2}\\
		\leq & O\left(h^{9}\kappa^{-1}\ev2+u^2h^{7}\kappa^{-1}\ef2+u^2h^{8}+ udh^{10}\kappa^{-1}+u^2h^{7}\kappa^{-1}\sigma^2\right)\\
 +& 2 \textbf{E} \bigg[\|\frac{uh}{2}\Ea{(1-e^{-2(h-\alpha h)})(\nabla f(x_\nh)-\nabla f(x_n^*(\alpha h)))}+\frac{uh}{2}\Ea{(1-e^{-2(h-\alpha h)})\nabla f(x_n^*(\alpha h))} \\
 - &\frac{u}{2}\int_{0}^{h}(1-e^{-2(h-s)})\nabla f(x_n^*(s))ds\|^2\bigg]\\
		\leq & O\left(h^{9}\kappa^{-1}\ev2+u^2h^{7}\kappa^{-1}\ef2+u^2h^{8}+ udh^{10}\kappa^{-1}+u^2h^{7}\kappa^{-1}\sigma^2\right)\\
		+&2u^2h^4\expec{\|\Ea{(\nabla f(x_\nh)-\nabla f(x_n^*(\alpha h)))}\|^2}\\
		\leq & O\left(h^{9}\kappa^{-1}\ev2+u^2h^{7}\kappa^{-1}\ef2+u^2h^{8}+ udh^{10}\kappa^{-1}+u^2h^{7}\kappa^{-1}\sigma^2\right)\\
		+&O\left(h^{10}\expec{\|v_n\|^2}+u^2(h^{12}+h^{11}\kappa^{-1})\expec{\|\nabla f(x_n)\|^2}+udh^{11}+u^2h^{11}\kappa^{-1}\sigma^2+u^2h^{12}\right)\\
		\leq & O\left((h^{10}+h^{9}\kappa^{-1})\expec{\|v_n\|^2}+u^2(h^{12}+h^{7}\kappa^{-1})\expec{\|\nabla f(x_n)\|^2}+ud(h^{11}+h^{10}\kappa^{-1})\right.\\
		&\left.+u^2h^{7}\kappa^{-1}\sigma^2+u^2h^{8}\right)
		\end{align*}
		The second inequality follows from Young's inequality, the third and fifth inequality uses the fact $1-e^{-2(\alpha-\alpha h)}\leq 2h$, the fifth inequality follows from the fact $\frac{uh}{2}\Ea{(1-e^{-2(h-\alpha h)})\nabla f(x_n^*(\alpha h))}-\frac{u}{2}\int_{0}^{h}(1-e^{-2(h-s)})\nabla f(x_n^*(s))ds=0$. 
		\item 	Next, note that
		\begin{align*}
		&\expec{\|\Ea{v_{n+1}}-v_n^*(h)\|^2} \\
		= &\expec{\|\Ea{uhe^{-2(h-\alpha h)}g_{\nu,b}(x_\nh)}-u\int_{0}^{h}e^{-2(h-s)}\nabla f(x_n^*(s))ds\|^2}\\
=& \textbf{E} \bigg[\|\Ea{uhe^{-2(h-\alpha h)}(g_{\nu,b}(x_\nh)-\nabla f(x_\nh)+\nabla f(x_\nh)-\nabla f(x_n^*(\alpha h))+\nabla f(x_n^*(\alpha h)))}\\
-&u\int_{0}^{h}e^{-2(h-s)}\nabla f(x_n^*(s))ds\|^2 \bigg]\\
		\leq & 2u^2h^2\expec{\|g_{\nu,b}(x_\nh)-\nabla f(x_\nh)\|^2}+2u^2h^2\expec{\|\nabla f(x_\nh)-\nabla f(x_n^*(\alpha h))\|^2}\\
		\leq & 2u^2h^2O\left(M^2h^{5}\kappa^{-1}\ev2+h^{3}\kappa^{-1}\ef2+h^{4}+ Mdh^{6}\kappa^{-1}+h^{3}\kappa^{-1}\sigma^2\right)\\
		+&2u^2h^2O\left(M^2h^6\expec{\|v_n\|^2}+(h^8+h^{7}\kappa^{-1})\expec{\|\nabla f(x_n)\|^2}+Mdh^7+h^{7}\kappa^{-1}\sigma^2+h^{8}\right)\\
		\leq & O\left((h^{7}\kappa^{-1}+h^{8})\ev2+u^2(h^{10}+h^{5}\kappa^{-1})\ef2+u^2h^{6}\right.\\
		&\left.+u^2h^{5}\kappa^{-1}\sigma^2+ud(h^{9}+h^{8}\kappa^{-1})\right)
		\end{align*}
		The first inequality follows from using, $\Ea{\nabla f(x_n^*(\alpha h)))}-u\int_{0}^{h}e^{-2(h-s)}\nabla f(x_n^*(s))ds=0$, and $e^{-2(h-\alpha h)}\leq 1$, and the second inequality follows from Lemma~\ref{lm:graddiffxnhalf}, and \ref{lm:graderrorxnhalf}.
		\item For the next part, note that we have
		\begin{align*}
		&\expec{\|x_{n+1}-x_n^*(h)\|^2}\\
		\leq & \expec{\|\frac{uh}{2}(1-e^{-2(h-\alpha h)})g_{\nu,b}(x_\nh)-\frac{u}{2}\int_{0}^{h}(1-e^{-2(h-s)})\nabla f(x_n^*(s))ds\|^2}\\
		\leq&\mathbf{E}\left[\|\frac{uh}{2}(1-e^{-2(h-\alpha h)})(g_{\nu,b}(x_\nh)-\nabla f(x_\nh)+\nabla f(x_\nh)-\nabla f(x_n^*(\alpha h))\right.\\
		&\left.+\nabla f(x_n^*(\alpha h)))-\frac{u}{2}\int_{0}^{h}(1-e^{-2(h-\alpha h)})\nabla f(x_n^*(s))ds+\frac{u}{2}\int_{0}^{h}(1-e^{-2(h-\alpha h)})\nabla f(x_n^*(s))ds\right.\\
		&\left.-\frac{u}{2}\int_{0}^{h}(1-e^{-2(h-s)})\nabla f(x_n^*(s))ds\|^2\right]\\
		\leq & 4u^2h^4\expec{\|g_{\nu,b}(x_\nh)-\nabla f(x_\nh)\|^2}+4u^2h^4\expec{\|\nabla f(x_\nh)-\nabla f(x_n^*(\alpha h))\|^2}\\
		+&\expec{\|uh(1-e^{-2(h-\alpha h)})\nabla f(x_n^*(\alpha h))-u\int_{0}^{h}(1-e^{-2(h-\alpha h)})\nabla f(x_n^*(s))ds\|^2}\\
		+&u^2\expec{\|\int_{0}^{h}(1-e^{-2(h-\alpha h)})\nabla f(x_n^*(s))ds-\int_{0}^{h}(1-e^{-2(h-s)})\nabla f(x_n^*(s))ds\|^2}\\
		\leq & O\left(h^{9}\kappa^{-1}\ev2+u^2h^{7}\kappa^{-1}\ef2+u^2h^{8}+ udh^{10}\kappa^{-1}+u^2h^{7}\kappa^{-1}\sigma^2\right)\\
		+& O\left(h^{10}\expec{\|v_n\|^2}+u^2(h^{12}+h^{11}\kappa^{-1})\expec{\|\nabla f(x_n)\|^2}+udh^{11}+u^2h^{11}\kappa^{-1}\sigma^2+u^2h^{12}\right)\\
		+& 16h^4\expec{\sup_{t\in[0,h]}\|x_n^*(0)-x_n^*(t)\|^2}+4u^2h^4\expec{\sup_{t\in[0,h]}\|\nabla f(x_n^*(t))\|^2}\\
		\leq & O\big((h^{10}+h^{9}\kappa^{-1})\ev2+u^2(h^{7}\kappa^{-1}+h^{12})\ef2\\
		+&u^2h^{8}+ (udh^{10}\kappa^{-1}+udh^{11})+u^2h^{7}\kappa^{-1}\sigma^2\big)\\
		+&O\left(h^{6}\ev2+u^2h^{8}\ef2+udh^{7}\right)\\
		+& O(h^6\ev2+u^2h^4\ef2+udh^7) \\
		\leq & O\big((h^{10}+h^{9}\kappa^{-1})\ev2+u^2(h^{7}\kappa^{-1}+h^{12})\ef2 \\
		+&u^2h^{8}+ (udh^{10}\kappa^{-1}+udh^{11})+u^2h^{7}\kappa^{-1}\sigma^2\big)\\
		+& O(h^6\ev2+u^2h^4\ef2+udh^7) \\
		\leq & O\left(h^{6}\ev2+u^2h^{4}\ef2+u^2h^{8}+ udh^7+u^2h^{7}\kappa^{-1}\sigma^2\right)
		\end{align*}
		The third inequality follows from Young's inequality, the fourth inequality follows from Lemma~\ref{lm:graddiffxnhalf}, and \ref{lm:graderrorxnhalf}, and the fact $1-e^{-2(\alpha-\alpha h)}\leq 2h$, and the fifth inequality follows from \eqref{eq:supx0xtdiff},and \eqref{eq:supgradsize}.
		\item Finally, note that we have
		\begin{align*}
		&\expec{\|v_{n+1}-v_n^*(h)\|^2}\\
		\leq & 2u^2h^4\expec{\|g_{\nu,b}(x_\nh)-\nabla f(x_\nh)\|^2}+O(h^4\ev2+u^2h^4\ef2+udh^5)\\
		\leq & O\left(h^{9}\kappa^{-1}\ev2+u^2h^{7}\kappa^{-1}\ef2+u^2h^{8}+ udh^{10}\kappa^{-1}+u^2h^{7}\kappa^{-1}\sigma^2\right)\\
		+&O(h^4\ev2+u^2h^4\ef2+udh^5)\\
		\leq & O\left(h^4\ev2+u^2h^4\ef2+u^2h^{8}+ udh^5+u^2h^{7}\kappa^{-1}\sigma^2\right)
		\end{align*}
		The first inequality follows from Lemma 2 of \cite{shen2019randomized}, and the second inequality follows from Lemma~\ref{lm:graderrorxnhalf}.
	\end{enumerate}
\end{proof}
\begin{lemma}\label{lm:contdiscdiff}
	Under conditions of Lemma~\ref{lm:lemma2eq},
	\begin{align*}
	\expec{f(x_{n+1}(0))-f(x_n(h))}\leq  O\left(Mh^5\ev2+uh^3\ef2+dh^6+uh^{4}\kappa^{-1}\sigma^2+uh^{5}\right) \numberthis\label{eq:contdiscdiff}
	\end{align*}
\end{lemma}
\begin{proof}[of Lemma~\ref{lm:contdiscdiff}]
Note that, we have 
	\begin{align*}
	&\expec{f(x_{n+1}(0))-f(x_n(h))}\\
	\leq & uh^3\expec{\|\nabla f(x_n(h))\|^2}+\frac{M}{h^3}\expec{\|\Ea{x_{n+1}(0)}-x_n(h)\|^2}+\frac{M}{2}\expec{\|x_{n+1}(0)-x_n(h)\|^2}\\
	\leq & uh^3O(M^2h^2\ev2+\ef2+Mdh^3)\\
	+&\frac{M}{h^3}O\left((h^{10}+h^{9}\kappa^{-1})\expec{\|v_n\|^2}+u^2(h^{12}+h^{7}\kappa^{-1})\expec{\|\nabla f(x_n)\|^2}+ud(h^{11}+h^{10}\kappa^{-1})+u^2h^{7}\kappa^{-1}\sigma^2+u^2h^{8}\right)\\
	+& \frac{M}{2} O\left(h^{6}\ev2+u^2h^{4}\ef2+u^2h^{8}+ udh^7+u^2h^{7}\kappa^{-1}\sigma^2\right)\\
	\leq & O\left(Mh^5\ev2+uh^3\ef2+dh^6+uh^{4}\kappa^{-1}\sigma^2+uh^{5}\right)
	\end{align*}
\end{proof}

\begin{lemma} \label{lm:sumvnsizebound}
	At iteration $n$, with the initial point $(x_n,v_n)$, for the updates \eqref{eq:rmpupdatexnh}, \eqref{eq:rmpupdatexn1}, and \eqref{eq:rmpupdatevn1}, we have
	\begin{align*}
	\sum_{n=0}^{N-1}\ev2\leq O\left(u^2h\sum_{n=0}^{N-1}\ef2+Ndu+Nu^2h^{3}\kappa^{-1}\sigma^2+Nu^2h^{4}\right) \numberthis\label{eq:sumvnsizebound}
	\end{align*}
\end{lemma}
\begin{proof}[of Lemma~\ref{lm:sumvnsizebound}]
	From Lemma 11 of \cite{shen2019randomized}, we have
	\begin{align*}
	&\expec{\frac{1}{2u}\|v_n(h)\|^2+f(x_n(h))}\\
	\leq & \expec{\frac{1}{2u}\|v_n\|^2+f(x_n)}-\frac{2}{3}hM\ev2+O\left(uh^3\ef2+dh\right)\numberthis\label{eq:potentialbound}
	\end{align*}
	From Lemma 11 we also have,
	\begin{align*}
	&\expec{\|v_{n+1}\|^2-\|v_n(h)\|^2}\\
	\leq & \frac{2}{h^2}\expec{\|v_{n+1}-v_n(h)\|^2}+4h^2\expec{\|v_n(h)\|^2}\\
	\leq & \frac{2}{h^2}O\left(h^4\ev2+u^2h^4\ef2+u^2h^{8}+ udh^5+u^2h^{7}\kappa^{-1}\sigma^2\right)\\
	+& 4h^2O(\ev2+u^2h^2\ef2+udh)\\
	\leq & O\left(h^2\ev2+u^2h^2\ef2+u^2h^{6}+ udh^3+u^2h^{5}\kappa^{-1}\sigma^2\right) \numberthis\label{eq:vn1vnhdiff}
	\end{align*}
	The second inequality above follows from \eqref{eq:supvsize}.
	From Lemma~\ref{lm:contdiscdiff}, we have
	\begin{align*}
	\expec{f(x_{n+1}(0))-f(x_n(h))}\leq  O\left(Mh^5\ev2+uh^3\ef2+dh^6+uh^{4}\kappa^{-1}\sigma^2+uh^{5}\right) 
	\end{align*}
	Now, from \eqref{eq:potentialbound}, \eqref{eq:vn1vnhdiff} and Lemma~\ref{lm:contdiscdiff}, we get
	\begin{align*}
	\expec{\frac{1}{2u}\|v_{n+1}\|^2+f(x_{n+1})}&=\expec{\frac{1}{2u}(\|v_{n+1}\|^2-\|v_n(h)\|^2)+f(x_{n+1})-f(x_n(h)}\\
	&+\expec{\frac{1}{2u}\|v_n(h)\|^2+f(x_n(h))}\\
	\leq & O\left(Mh^2\ev2+uh^2\ef2+uh^{6}+ dh^3+uh^{5}\kappa^{-1}\sigma^2\right)\\
	+& O\left(Mh^5\ev2+uh^3\ef2+dh^6+uh^{4}\kappa^{-1}\sigma^2+uh^{5}\right)\\
	+&\expec{\frac{1}{2u}\|v_n\|^2+f(x_n)}-\frac{2}{3}hM\ev2+O\left(uh^3\ef2+dh\right)
	\end{align*}
	Choosing $h$ such that, $\frac{1}{3}hM\geq Mh^2$, i.e., $ h\leq \frac{1}{3}$, we get
	\begin{align*}
	&\expec{\frac{1}{2u}\|v_{n+1}\|^2+f(x_{n+1})}\\
	\leq & O\left(uh^2\ef2+dh+uh^{4}\kappa^{-1}\sigma^2+uh^{5}\right)
	+\expec{\frac{1}{2u}\|v_n\|^2+f(x_n)}-\frac{1}{3}hM\ev2
	\end{align*}
	Summing both sides from $n=0$ to $N-1$, we get
	\begin{align*}
	\sum_{n=0}^{N-1}\expec{\frac{1}{2u}\|v_{n+1}\|^2+f(x_{n+1})}
	\leq & O\left(uh^2\sum_{n=0}^{N-1}\ef2+Ndh+Nuh^{4}\kappa^{-1}\sigma^2+Nuh^{5}\right)\\
	+&\expec{\frac{1}{2u}\sum_{n=0}^{N-1}\left(\|v_n\|^2+f(x_n)\right)}-\frac{1}{3}hM\sum_{n=0}^{N-1}\ev2
	\end{align*}
	Since, $\|v_0\|=0$, and $\expec{f(x_0)}-f(x^*)\leq O(d)$, and consequently, $\expec{f(x_0)-f(x_N)}\leq O(d)$, we have
	\begin{align*}
	&\frac{1}{3}hM\sum_{n=0}^{N-1}\ev2\leq O\left(uh^2\sum_{n=0}^{N-1}\ef2+Ndh+Nuh^{4}\kappa^{-1}\sigma^2+Nuh^{5}\right)\\
	&\sum_{n=0}^{N-1}\ev2\leq O\left(u^2h\sum_{n=0}^{N-1}\ef2+Ndu+Nu^2h^{3}\kappa^{-1}\sigma^2+Nu^2h^{4}\right)
	\end{align*}
\end{proof}
\begin{lemma}\label{lm:gradfnvnsuminnerprodbound}
	At iteration $n$, with the initial point $(x_n,v_n)$, for the updates \eqref{eq:rmpupdatexnh}, \eqref{eq:rmpupdatexn1}, and \eqref{eq:rmpupdatevn1}, we have
	\begin{align*}
	&\sum_{n=0}^{N-1}\expec{\|\nabla f(x_n)\|^2}\leq O\left(\frac{M}{h}\left\lvert\expec{\nabla f(x_N)^\top v_N}\right\rvert+MNd+Nh^{3}\kappa^{-1}\sigma^2+Nh^{4}\right)\\
	&\sum_{n=0}^{N-1}\ev2\leq O\left(u\left\lvert\expec{\nabla f(x_N)^\top v_N}\right\rvert+Ndu+Nu^2h^{3}\kappa^{-1}\sigma^2+Nu^2h^{4}\right)
	\end{align*}
\end{lemma}
\begin{proof}[of Lemma~\ref{lm:gradfnvnsuminnerprodbound}]
	From (15) in Lemma 12 of \cite{shen2019randomized} we have,
	\begin{align*}
	&\expec{\nabla f(x_n(h))^\top v_n(h)}\\
	\leq & \expec{\nabla f(x_n)^\top v_n}-\frac{1}{6}uh\ef2+O\left(Mh\ev2+uh^3\ef2+dh^2\right)\numberthis\label{eq:expecnablaxnvn}
	\end{align*}
	From Lemma 12 of \cite{shen2019randomized} we also have,
	\begin{align*}
	&\expec{\nabla f(x_{n+1})^\top v_{n+1}-\nabla f(x_n(h))^\top v_n(h)}\\
	\leq & \frac{2u}{h}\expec{\|\nabla f(x_{n+1})-\nabla f(x_n(h))\|^2}+\frac{2M}{h^2}\expec{\|v_{n+1}-v_n(h)\|^2}+uh^2\expec{\|\nabla f(x_n(h))\|^2}+Mh\expec{\|v_n(h)\|^2}\\
	\leq & \frac{2M}{h}\expec{\|x_{n+1}-x_n(h)\|^2}+\frac{2M}{h^2}\expec{\|v_{n+1}-v_n(h)\|^2}+uh^2\expec{\|\nabla f(x_n(h))\|^2}+Mh\expec{\|v_n(h)\|^2}
	\end{align*}
	Now from \eqref{eq:xn1xnhdiff}, \eqref{eq:vn1vnhnormdiff}, and Lemma~\ref{lm:contdiffprops} we have,
	\begin{align*}
	&\expec{\nabla f(x_{n+1})^\top v_{n+1}-\nabla f(x_n(h))^\top v_n(h)}\\
	\leq & \frac{2M}{h}O\left(h^{6}\ev2+u^2h^{4}\ef2+u^2h^{8}+ udh^7+u^2h^{7}\kappa^{-1}\sigma^2\right) \\
	+& \frac{2M}{h^2}O\left(h^4\ev2+u^2h^4\ef2+u^2h^{8}+ udh^5+u^2h^{7}\kappa^{-1}\sigma^2\right)\\
	+&uh^2O(M^2h^2\|v_n\|^2+\|\nabla f(x_n)\|^2+Mdh^3)
	+MhO(\|v_n\|^2+u^2h^2\|\nabla f(x_n)\|^2+udh)\\
	\leq& O\left(Mh\ev2+uh^{2}\ef2+dh^2+uh^{6}+ uh^{5}\kappa^{-1}\sigma^2\right)\numberthis\label{eq:expecnablaxn1vn1expecnablaxnvn}
	\end{align*}
	Combining \eqref{eq:expecnablaxnvn}, and \eqref{eq:expecnablaxn1vn1expecnablaxnvn}, we get
	\begin{align*}
	\expec{\nabla f(x_{n+1})^\top v_{n+1}}\leq &  \expec{\nabla f(x_n)^\top v_n}-\frac{1}{6}uh\ef2 \\
	+ &O\left(Mh\ev2+uh^{2}\ef2+dh^2+uh^{6}+ uh^{5}\kappa^{-1}\sigma^2\right).
	\end{align*}
	Summing both sides from $n=0$ to $N-1$, and using Lemma~\ref{lm:sumvnsizebound}, we get
	\begin{align*}
	&\sum_{n=0}^{N-1}\expec{\nabla f(x_{n+1})^\top v_{n+1}}\\
	\leq & \sum_{n=0}^{N-1} \expec{\nabla f(x_n)^\top v_n}-\frac{1}{6}uh\sum_{n=0}^{N-1}\ef2\\
	+& O\left(Mh\sum_{n=0}^{N-1}\ev2+uh^{2}\sum_{n=0}^{N-1}\ef2+Ndh^2+Nuh^{6}+ Nuh^{5}\kappa^{-1}\sigma^2\right)\\
	\leq & \sum_{n=0}^{N-1} \expec{\nabla f(x_n)^\top v_n}-\frac{1}{6}uh\sum_{n=0}^{N-1}\ef2\\
	+& O\left(MhO\left(u^2h\sum_{n=0}^{N-1}\ef2+Ndu+Nu^2h^{3}\kappa^{-1}\sigma^2+Nu^2h^{4}\right)\right.\\
	&\left.+uh^{2}\sum_{n=0}^{N-1}\ef2+Ndh^2+Nuh^{6}+ Nuh^{5}\kappa^{-1}\sigma^2\right)\\
	\leq & \sum_{n=0}^{N-1} \expec{\nabla f(x_n)^\top v_n}-\frac{1}{6}uh\sum_{n=0}^{N-1}\ef2\\
	+& O\left(uh^2\sum_{n=0}^{N-1}\ef2+Ndh+Nuh^{4}\kappa^{-1}\sigma^2+Nuh^{5}\right)
	\end{align*}
	Now choosing $\frac{1}{24}uh\geq uh^2$, and $v_0=0$, we have,
	\begin{align*}
	&\frac{1}{8}uh\sum_{n=0}^{N-1}\ef2\leq O\left(\left\lvert\expec{\nabla f(x_N)^\top v_N}\right\rvert+Ndh+Nuh^{4}\kappa^{-1}\sigma^2+Nuh^{5}\right)\\
	&\sum_{n=0}^{N-1}\ef2\leq O\left(\frac{M}{h}\left\lvert\expec{\nabla f(x_N)^\top v_N}\right\rvert+MNd+Nh^{3}\kappa^{-1}\sigma^2+Nh^{4}\right)
	\end{align*}
	Using Lemma~\ref{lm:sumvnsizebound}, we have,
	\begin{align*}
	\sum_{n=0}^{N-1}\ev2\leq& O\left(u\left\lvert\expec{\nabla f(x_N)^\top v_N}\right\rvert+Ndu+Nu^2h^{3}\kappa^{-1}\sigma^2+Nu^2h^{4}\right)
	\end{align*}
\end{proof}
\begin{proof}[of Theorem~\ref{th:rmpmainresult}]
	From Theorem 3 in \cite{shen2019randomized}, we have
	\begin{align*}
	q_N\leq & e^{-\frac{Nh}{2\kappa}}q_0+\sum_{n=1}^{N}\frac{2\kappa}{h}\left(2\expec{\|\Ea{v_{n+1}}-v_n^*(h)\|^2}+3\expec{\|\Ea{x_{n+1}}-x_n^*(h)\|^2}\right)\\
	+&\sum_{n=1}^{N}\left(2\expec{\|{v_{n+1}}-v_n^*(h)\|^2}+3\expec{\|{x_{n+1}}-x_n^*(h)\|^2}\right) \numberthis\label{eq:qNmother}
	\end{align*}
	where $q_N=\expec{\|x_N-y_N\|^2+\|x_N+v_N-y_N-w_N\|^2}$.
	We also have,
	\begin{align*}
	e^{-\frac{Nh}{2\kappa}}q_0\leq \frac{\epsilon^2d}{4m}\numberthis\label{eq:q0bound}
	\end{align*} 
	From Lemma~\ref{lm:lemma2eq},
	\begin{align*}
	&\sum_{n=1}^{N}\frac{2\kappa}{h}\left(2\expec{\|\Ea{v_{n+1}}-v_n^*(h)\|^2}+3\expec{\|\Ea{x_{n+1}}-x_n^*(h)\|^2}\right)\\
	\leq & O\left((h^{6}+\kappa h^{7})\sum_{n=1}^{N}\ev2+u^2(\kappa h^{9}+h^{4})\sum_{n=1}^{N}\ef2\right.\\
	&\left.+N\kappa u^2h^{5}+ Nu^2h^{4}\sigma^2+Nud(\kappa h^{8}+h^{7})\right)\numberthis\label{eq:vn1alphaxn1alpha}
	\end{align*}
	\begin{align*}
	&\sum_{n=1}^{N}\left(2\expec{\|{v_{n+1}}-v_n^*(h)\|^2}+3\expec{\|{x_{n+1}}-x_n^*(h)\|^2}\right)\\
	\leq & \sum_{n=1}^{N}\left(O\left(h^4\ev2+u^2h^4\ef2+u^2h^{8}+ udh^5+u^2h^{7}\kappa^{-1}\sigma^2\right)\right.\\
	&\left. O\left(h^{6}\ev2+u^2h^{4}\ef2+u^2h^{8}+ udh^7+u^2h^{7}\kappa^{-1}\sigma^2\right)\right)\\
	\leq & O\left(h^{4}\sum_{n=1}^{N}\ev2+u^2h^{4}\sum_{n=1}^{N}\ef2+Nu^2h^{8}+ Nudh^5+Nu^2h^{7}\kappa^{-1}\sigma^2\right) \numberthis\label{eq:vn1xn1}
	\end{align*}
	Combining \eqref{eq:vn1alphaxn1alpha}, and \eqref{eq:vn1xn1}, we have
	\begin{align*}
	&\sum_{n=1}^{N}\frac{2\kappa}{h}\left(2\expec{\|\Ea{v_{n+1}}-v_n^*(h)\|^2}+3\expec{\|\Ea{x_{n+1}}-x_n^*(h)\|^2}\right)\\
	+&\sum_{n=1}^{N}\left(2\expec{\|{v_{n+1}}-v_n^*(h)\|^2}+3\expec{\|{x_{n+1}}-x_n^*(h)\|^2}\right)\\
	\leq & O\left((h^{4}+\kappa h^7)\sum_{n=1}^{N}\ev2+u^2(h^{4}+\kappa h^9+h^{4})\sum_{n=1}^{N}\ef2\right.\\
	&\left.+Nudh^5+Nud(\kappa h^{8}+h^{7})+Nu^2h^{4}\sigma^2+N\kappa u^2h^{5}\right) \numberthis\label{eq:theorem3intermed}
	\end{align*}
	From the proof of Theorem 3 of \cite{shen2019randomized} we have,
	\begin{align*}
	\left\lVert\expec{\nabla f(x_N)^\top v_N}\right\rVert\leq 4d + 6Mq_N
	\end{align*}
	Then we have,
	\begin{align*}
	\sum_{n=0}^{N-1}\ev2\leq& O\left(q_N+Ndu+Nu^2h^{3}\kappa^{-1}\sigma^2+Nu^2h^{4}\right)
	\numberthis\label{eq:finalvn2sizebound}
	\end{align*}
	and
	\begin{align*}
	\sum_{n=0}^{N-1}\ef2\leq O\left(\frac{dM}{h}+\frac{M^2}{h}q_N+MNd+Nh^{3}\kappa^{-1}\sigma^2+Nh^{4}\right)\numberthis\label{eq:finalgradsizebound}
	\end{align*}
	From \eqref{eq:theorem3intermed}, \eqref{eq:finalvn2sizebound}, and \eqref{eq:finalgradsizebound}, and setting $N=\frac{2\kappa}{h}\log\left(\frac{20}{\epsilon^2}\right)$ we have
	\begin{align*}
	&\sum_{n=1}^{N}\frac{2\kappa}{h}\left(2\expec{\|\Ea{v_{n+1}}-v_n^*(h)\|^2}+3\expec{\|\Ea{x_{n+1}}-x_n^*(h)\|^2}\right)\\
	+&\sum_{n=1}^{N}\left(2\expec{\|{v_{n+1}}-v_n^*(h)\|^2}+3\expec{\|{x_{n+1}}-x_n^*(h)\|^2}\right)\\
	\leq & O\left((h^3+\kappa h^7)q_N+Ndu(h^4+\kappa h^7)+duh^3+duh^{3}+Nu^2h^{4}\sigma^2+Nu^2\kappa h^{5}\right)\\
	\leq & O\left((h^3+\kappa h^7)q_N+\frac{d}{m}(h^3+\kappa h^6)\log\left(\frac{1}{\epsilon}\right)+\frac{h^{3}\kappa^{-1}}{m^2}\sigma^2\log\left(\frac{1}{\epsilon}\right)+\frac{h^{4}}{m^2}\log\left(\frac{1}{\epsilon}\right)\right)
	\end{align*}
	From \eqref{eq:qNmother}, and \eqref{eq:q0bound},
	\begin{align*}
	q_N\leq \frac{\epsilon^2d}{4m}+O\left((h^3+\kappa h^7)q_N+\left(\frac{d}{m}(h^3+\kappa h^6)+\frac{h^3}{Mm}\sigma^2+\frac{h^4}{m^2}\right)\log\left(\frac{1}{\epsilon}\right)\right)
	\end{align*}
	Using, $(h^3+\kappa h^7)\leq 1/2$, we have,
	\begin{align*}
	\frac{q_N}{2}\leq \frac{\epsilon^2d}{4m}+O\left(\left(\frac{d}{m}(h^3+\kappa h^6)+\frac{h^3}{Mm}\sigma^2+\frac{h^4}{m^2}\right)\log\left(\frac{1}{\epsilon}\right)\right)
	\end{align*}
	Choosing, $h=C\min\left(\frac{(\epsilon\sqrt{m})^\frac{1}{3}}{(d\kappa)^\frac{1}{6}\log\left(\frac{1}{\epsilon}\right)^\frac{1}{6}},\min\left(\left(\frac{m}{d}\right)^\frac{1}{3},\left(\frac{Mm}{16\sigma^2}\right)^\frac{1}{3},\sqrt{m}\right)\epsilon^\frac{2}{3}\log\left(\frac{1}{\epsilon}\right)^{-\frac{2}{3}}\right)$, we get,
	\begin{align*}
	\expec{\|x_N-y_N\|^2}\leq q_N\leq \frac{\epsilon^2d}{m}
	\end{align*}
	So, the iteration complexity is given by,
	\begin{align*}
	N=\tilde{O}\left(\max\left(\frac{d^\frac{1}{6}\kappa^\frac{7}{6}}{(\epsilon\sqrt{m})^\frac{1}{3}},\frac{\kappa\max\left(\left(\frac{d}{m}\right)^\frac{1}{3},\left(\frac{\sigma^2}{Mm}\right)^\frac{1}{3},\frac{1}{\sqrt{m}}\right)}{\epsilon^\frac{2}{3}}\right)\right)
	\end{align*}
	The total number of zeroth-order oracle calls are given by,
	\begin{align*}
	Nb= \tilde{O}\left(\max\left(\frac{d^\frac{5}{3}\kappa^\frac{8}{3}}{\epsilon^\frac{4}{3}},\frac{d\kappa^2\max\left(\left(\frac{d}{m}\right)^\frac{1}{3},\left(\frac{\sigma^2}{Mm}\right)^\frac{1}{3},\frac{1}{\sqrt{m}}\right)^4}{\epsilon^\frac{8}{3}}\right)\right)
	\end{align*}
\end{proof}

\section{Proofs for Section~\ref{as:lsi}}

\begin{proof}[of Theorem~\ref{th:lsilmc}]
	Let us define the following continuous time SDE with the initial point $\hx_0$:
	\begin{align*}
	\hat x_t=-g_{\nu,b}(\hat x_0)dt+\sqrt{2}dW_n
	\end{align*}
	Observe that $\hat x_h$ has the same distribution as $x_{n+1}$ when $\hat x_0=x_n$. Let $z$ denote $(\{u_i\}_{i=1}^{b},\{\xi_i\}_{i=1}^{b}\})$. To show the dependence of $g_{\nu,b}(\hat \hx_0)$ on $z$ we will use $g_{\nu,b}(\hat x_0,z)$ to denote $g_{\nu,b}(\hat x_0)$ just for this proof. Let $\rho_{t0z}(\hx_t,\hx_0,z)$ be the joint distribution of $\hx_t$, $\hx_0$, and $z$. Observe that conditioned on $\hx_0$, and $z$, $g_{\nu,b}(\hat x_0,z)$ is deterministic. Then by Fokker-Plank equation, we have 
	\begin{align*}	
	\pd{\rho_{t|0,z}(\hx_t|\hx_0,z)}{t}=\nabla\cdot \left(\rho_{t|0,z}(\hx_t|\hx_0,z)g_{\nu,b}(\hat \hx_0,z)\right)+\Delta \rho_{t|0,z}(\hx_t|\hx_0,z)
	\end{align*}
	Then the time evolution of $\rho_t(x)$ is given by
	\begin{align*}
	&\pd{\rho_t(x)}{t}=\expect{x_0,z}{\pd{\rho_{t|0,z}(x|x_0,z)}{t}}\\
	\leq & \expect{x_0,z}{\nabla\cdot \left(\rho_{t|0,z}(\hx_t|\hx_0,z)g_{\nu,b}(\hat \hx_0,z)\right)+\Delta \rho_{t|0,z}(\hx_t|\hx_0,z)}\\
	=&\expect{x_0,z}{\nabla\cdot \left(\rho_{t|0,z}(\hx_t|\hx_0,z)g_{\nu,b}(\hat \hx_0,z)\right)}+\Delta \rho_{t}(\hx_t)\\
	=&\int_{\mathbb{R}^d}\int_{\mathbb{R}^{2d}}\nabla\cdot \left(\rho_{t,0,z}(\hx_t,\hx_0,z)g_{\nu,b}(\hat \hx_0,z)\right)dx_0dz+\Delta \rho_{t}(\hx_t)\\
	=&\nabla\cdot\left(\rho_t(x)\expect{0,z|t}{g_{\nu,b}(\hx_0,z)|\hx_t=x}\right)+\Delta \rho_{t}(\hx_t) \numberthis\label{eq:rhottimeevol}
	\end{align*}
	Now, as shown in \cite{vempala2019rapid} we have,
	\begin{align*}
	\pd{H_\pi(\rho_t(x))}{t}=\int_{\mathbb{R}^d}\pd{\rho_t(x)}{t}\log \left(\frac{\rho_t(x)}{\pi(x)}\right)dx
	\end{align*}
	Then using \eqref{eq:rhottimeevol}, we have
	\begin{align*}
	&\pd{H_\pi(\rho_t(x))}{t}\\
	=&\int_{\mathbb{R}^d}\left(\nabla\cdot\left(\rho_t(x)\expect{0,z|t}{g_{\nu,b}(\hx_0,z)|\hx_t=x}\right)+\Delta \rho_{t}(\hx_t)\right)\log \left(\frac{\rho_t(x)}{\pi(x)}\right)dx\\
	=&\int_{\mathbb{R}^d}\nabla\cdot\left(\left(\rho_t(x)\expect{0,z|t}{g_{\nu,b}(\hx_0,z)|\hx_t=x}\right)+\nabla \rho_{t}(\hx_t)\right)\log \left(\frac{\rho_t(x)}{\pi(x)}\right)dx\\
	=&\int_{\mathbb{R}^d}\nabla\cdot\left(\rho_t(x)\left(\nabla\log \left(\frac{\rho_t(x)}{\pi(x)}\right)+\expect{0,z|t}{g_{\nu,b}(\hx_0,z)|\hx_t=x}-\nabla f(x)\right)\right)\log \left(\frac{\rho_t(x)}{\pi(x)}\right)dx
	\end{align*}
	Now we use the fact that $\nabla\cdot(ax)=ax\cdot\nabla a+a\nabla\cdot x$ where $a$ is a scalar, and $x$ is a vector:
	\begin{align*}
	&\pd{H_\pi(\rho_t(x))}{t}\\
	=&\int_{\mathbb{R}^d}\nabla\cdot\left(\rho_t(x)\left(\nabla\log \left(\frac{\rho_t(x)}{\pi(x)}\right)+\expect{0,z|t}{g_{\nu,b}(\hx_0,z)|\hx_t=x}-\nabla f(x)\right)\log \left(\frac{\rho_t(x)}{\pi(x)}\right)\right)dx\\
	-& \int_{\mathbb{R}^d}\rho_t(x)\left\langle\left(\nabla\log \left(\frac{\rho_t(x)}{\pi(x)}\right)+\expect{0,z|t}{g_{\nu,b}(\hx_0,z)|\hx_t=x}-\nabla f(x)\right),\nabla\log \left(\frac{\rho_t(x)}{\pi(x)}\right)\right\rangle dx
	\end{align*}
	Now \textcolor{black}{as $\rho_t(x)\left(\nabla\log \left(\frac{\rho_t(x)}{\pi(x)}\right)+\expect{0,z|t}{g_{\nu,b}(\hx_0,z)|\hx_t=x}-\nabla f(x)\right)\log \left(\frac{\rho_t(x)}{\pi(x)}\right)$ decays to $0$ as $x$ goes to infinity,} we have, 
	\begin{align*}
	\int_{\mathbb{R}^d}\nabla\cdot\left(\rho_t(x)\left(\nabla\log \left(\frac{\rho_t(x)}{\pi(x)}\right)+\expect{0,z|t}{g_{\nu,b}(\hx_0,z)|\hx_t=x}-\nabla f(x)\right)\log \left(\frac{\rho_t(x)}{\pi(x)}\right)\right)dx=0
	\end{align*}
	Then we get, 
	\begin{align*}
	&\pd{H_\pi(\rho_t(x))}{t}\\
	=&-\int_{\mathbb{R}^d}\rho_t(x)\left\langle\left(\nabla\log \left(\frac{\rho_t(x)}{\pi(x)}\right)+\expect{0,z|t}{g_{\nu,b}(\hx_0,z)|\hx_t=x}-\nabla f(x)\right),\nabla\log \left(\frac{\rho_t(x)}{\pi(x)}\right)\right\rangle dx\\
	=&-J_\pi(\rho_{t}(x))-\int_{\mathbb{R}^d}\int_{\mathbb{R}^d}\int_{\mathbb{R}^{2d}}\rho_t(x,\hx_0,z)\left\langle g_{\nu,b}(\hx_0,z)-\nabla f(x),\nabla\log \left(\frac{\rho_t(x)}{\pi(x)}\right)\right\rangle dzdxd\hx_0\\
	=&-J_\pi(\rho_{t}(x))+\expect{t0z}{\left\langle \nabla f(\hx_t)- g_{\nu,b}(\hx_0,z),\nabla\log \left(\frac{\rho_t(x)}{\pi(x)}\right)\right\rangle}\numberthis\label{eq:entropyevolboundintermediate}
	\end{align*}
	The second equality above follows from \eqref{eq:HJdef}, and in the last line we have substituted $x_t$ in place of $x$. Now we will upper bound the second term above. 
	\begin{align*}
	&\expect{t0z}{\left\langle \nabla f(\hx_t)- g_{\nu,b}(\hx_0,z),\nabla\log \left(\frac{\rho_t(x)}{\pi(x)}\right)\right\rangle}\\
	\leq & \expect{t0z}{\| \nabla f(\hx_t)- g_{\nu,b}(\hx_0,z)\|^2}+\frac{1}{4}\expect{t0z}{\left\|\nabla\log \left(\frac{\rho_t(x)}{\pi(x)}\right)\right\|^2}\\
	\leq & 2M^2\expect{t0}{\| \hx_t- \hx_0\|^2}+2\expect{0z}{\| \nabla f(\hx_0)- g_{\nu,b}(\hx_0,z)\|^2}+\frac{1}{4}J_{\pi}(\rho_t(x)) \numberthis\label{eq:expecinnerprodbound}
	\end{align*}
	Now, from Lemma~\ref{lm:zogradvar}, we have,
	\begin{align*}
	\expect{0z}{\| \nabla f(\hx_0)- g_{\nu,b}(\hx_0,z)\|^2}
	\leq  \frac{4(d+5)\expect{0z}{\|\nabla f(\hx_0)\|^2}}{b}+C_1\numberthis\label{eq:zolmclsigradvar}
	\end{align*}
	where $C_1=\frac{4(d+5)\sigma^2}{b}+\frac{3\nu^2M^2(d+3)^3}{2}$. We also have, with $\tau_0\sim N(0,\mathbf{I_d})$
	\begin{align*}
	&\expect{t0}{\| \hx_t- \hx_0\|^2}\\
	=&\expect{t0}{\| -tg_{\nu,b}(\hx_0,z)+\sqrt{2t}\tau_0\|^2}\\
	\leq & 2dt+2t^2\expect{0z}{\| \nabla f(\hx_0)- g_{\nu,b}(\hx_0,z)\|^2}+2t^2\expect{0}{\|\nabla f(\hx_0)\|^2} \numberthis\label{eq:expecxtx0boundlmclsi}
	\end{align*}
	Combining \eqref{eq:expecinnerprodbound}, \eqref{eq:zolmclsigradvar}, and \eqref{eq:expecxtx0boundlmclsi}, for $t\leq 1/(2M)$ we get
	\begin{align*}
	&\expect{t0z}{\left\langle \nabla f(\hx_t)- g_{\nu,b}(\hx_0,z),\nabla\log \left(\frac{\rho_t(x)}{\pi(x)}\right)\right\rangle}\\
	\leq & \frac{1}{4}J_{\pi}(\rho_t(x))+4M^2td+(2+4M^2t^2)\expect{0z}{\| \nabla f(\hx_0)- g_{\nu,b}(\hx_0,z)\|^2}+4M^2t^2\expect{0}{\|\nabla f(\hx_0)\|^2}  \\
	\leq & \frac{1}{4}J_{\pi}(\rho_t(x))+4M^2td+3\left(\frac{4(d+5)\expect{0}{\|\nabla f(\hx_0)\|^2}}{b}+C_1\right)+4M^2t^2\expect{0}{\|\nabla f(\hx_0)\|^2}  \\
	\leq & \frac{1}{4}J_{\pi}(\rho_t(x))+4M^2td+3C_1+\left(\frac{12(d+5)}{b}+4M^2t^2\right)\expect{0}{\|\nabla f(\hx_0)\|^2} \\
	\leq & \frac{1}{4}J_{\pi}(\rho_t(x))+4M^2td+3C_1+\left(\frac{12(d+5)}{b}+4M^2t^2\right)\left(\frac{4M^2}{\lambda}H_{\pi}(\rho_0(x))+2Md\right) \numberthis\label{eq:entropyevolsecondtermbound}
	\end{align*}
	We get the last inequality using Lemma 12 of \cite{vempala2019rapid}. Now combining, \eqref{eq:entropyevolboundintermediate}, and \eqref{eq:entropyevolsecondtermbound}, we get,
	\begin{align*}
	&\pd{H_\pi(\rho_t(x))}{t}\\
	\leq & -\frac{3}{4}J_{\pi}(\rho_t(x))+4M^2td+3C_1+\left(\frac{12(d+5)}{b}+4M^2t^2\right)\left(\frac{4M^2}{\lambda}H_{\pi}(\rho_0(x))+2Md\right) 
	\end{align*}
	Using \eqref{eq:JHrel}, we get
	\begin{align*}
	&\pd{H_\pi(\rho_t(x))}{t}\\
	\leq & -\frac{3\lambda}{2}H_{\pi}(\rho_t(x))+4M^2td+3C_1+\left(\frac{12(d+5)}{b}+4M^2t^2\right)\left(\frac{4M^2}{\lambda}H_{\pi}(\rho_0(x))+2Md\right)\numberthis\label{eq:entropyevolbound}
	\end{align*}
	Taking $t\leq h$, we get,
	\begin{align*}
	&\pd{H_\pi(\rho_t(x))}{t}\\
	\leq & -\frac{3\lambda}{2}H_{\pi}(\rho_t(x))+4M^2hd+3C_1+\left(\frac{12(d+5)}{b}+4M^2h^2\right)\left(\frac{4M^2}{\lambda}H_{\pi}(\rho_0(x))+2Md\right)
	\end{align*}
	Multiplying both sides with $e^{\frac{3\lambda t}{2}}$, and integrating from $t=0$ to $h$, we get
	\begin{align*}
	&e^\frac{3\lambda h}{2}H_\pi(\rho_h(x))-H_\pi(\rho_0(x))\\
	\leq & \frac{2(e^{\frac{3\lambda h}{2}}-1)}{3\lambda}\left(4M^2hd+3C_1+\left(\frac{12(d+5)}{b}+4M^2h^2\right)\left(\frac{4M^2}{\lambda}H_{\pi}(\rho_0(x))+2Md\right)\right)\\
	\leq & 2h\left(4M^2hd+3C_1+\left(\frac{12(d+5)}{b}+4M^2h^2\right)\left(\frac{4M^2}{\lambda}H_{\pi}(\rho_0(x))+2Md\right)\right)\\
	=&\left(8M^2h^2d+3hC_1+\left(\frac{24(d+5)h}{b}+8M^2h^3\right)Md\right)+\frac{4M^2}{\lambda}\left(\frac{24(d+5)h}{b}+8M^2h^3\right)H_{\pi}(\rho_0(x))
	\end{align*}
	As in \cite{vempala2019rapid}, in the penultimate step we use the fact $e^a\leq 1+2a$ for $0<a=\frac{3\lambda h}{2}$, and $h\leq \frac{2}{3\lambda}$. Hence, we have
	\begin{align*}
	H_\pi(\rho_h(x))&\leq e^{-\frac{3\lambda h}{2}}\left(1+\frac{4M^2}{\lambda}\left(\frac{24(d+5)h}{b}+8M^2h^3\right)\right)H_\pi(\rho_0(x)) \\
	+&e^{-\frac{3\lambda h}{2}}\left(8M^2h^2d+3hC_1+\left(\frac{24(d+5)h}{b}+8M^2h^3\right)Md\right).
	\end{align*}
	Choosing $b\geq \frac{384 M^2(d+5)}{\lambda^2}$, and $h\leq \frac{\lambda }{12M^2}$, we get, $$1+\frac{4M^2}{\lambda}\left(\frac{24(d+5)h}{b}+8M^2h^3\right)\leq 1+\frac{\lambda h}{2}\leq e^{\frac{\lambda h}{2}}.$$
	Then we have,
	\begin{align*}
	H_\pi(\rho_h(x))\leq e^{-\lambda h}H_\pi(\rho_0(x))+\left(8M^2h^2d+3hC_1+\left(\frac{24(d+5)h}{b}+8M^2h^3\right)Md\right)
	\end{align*}
	Observe that when $\hx_0=x_n$, $\rho_0$ is same as $\varpi_{n}$, and then $\rho_h(x)$ is same as $\varpi_{n+1}$. Then
	\begin{align*}
	&H_\pi(\varpi_{n+1})\\
	\leq& e^{-\lambda h}H_\pi(\varpi_{n})+\left(8M^2h^2d+3hC_1+\left(\frac{24(d+5)h}{b}+8M^2h^3\right)Md\right)\\
	\leq &e^{-(n+1)\lambda h}H_\pi(\varpi_{0})+\frac{1}{1-e^{-\lambda h}}\left(8M^2h^2d+3hC_1+\left(\frac{24(d+5)h}{b}+8M^2h^3\right)Md\right)
	\end{align*}
	Choosing $n=N=\frac{1}{\lambda h}\log\left(\frac{\epsilon^2}{H_\pi(\varpi_{0})}\right)$, and using $1-e^{-\lambda h}\geq \frac{\lambda h}{2}$, for $h\leq \frac{1}{\lambda}$, we get
	\begin{align*}
	&H_\pi(\varpi_{N})\\
	\leq & \epsilon^2 +\left(\frac{16M^2hd}{\lambda}+\frac{6}{\lambda}\left(\frac{4(d+5)\sigma^2}{b}+\frac{3\nu^2M^2(d+3)^3}{2}\right)+\left(\frac{48(d+5)}{b}+16M^2h^2\right)\frac{Md}{\lambda}\right)
	\end{align*}
	Choosing $b$, $\nu$, and $h$ as in \eqref{eq:lsilmcparamchoice}, we get 
	\begin{align*}
	H_\pi(\varpi_{N})=O(\epsilon^2)
	\end{align*}
	Using, \eqref{eq:w2hinequality}, we get,
	\begin{align*}
	W_2(\varpi_N,\pi)=O(\epsilon)
	\end{align*}
\end{proof}

\section{Proofs for Section~\ref{sec:noise}}

\begin{proof}[of Lemma~\ref{lm:zogradesterrordiffnoise}]
First note that, we have
	\begin{align*}
	g_{\nu,b}(\theta)-\nabla f_\nu(\theta)=&\frac1b\sum_{i=1}^{b}\frac{F(\theta+\nu u_i,\xi_i)-F(\theta,\xi_i')}{\nu}u_i-\nabla f_\nu(\theta)\\
	=&\frac1b\sum_{i=1}^{b}\frac{f(\theta+\nu u_i)-f(\theta)}{\nu}u_i-\nabla f_\nu(\theta)+\frac1b\sum_{i=1}^{b}\frac{\xi_i-\xi_i'}{\nu}u_i.
\end{align*}
Hence, we have
\begin{align*}
	\expec{\left\|g_{\nu,b}(\theta)-\nabla f_\nu(\theta)\right\|^2}=&\expec{\left\|\frac1b\sum_{i=1}^{b}\frac{f(x+\nu u_i)-f(\theta)}{\nu}u_i-\nabla f_\nu(\theta)\right\|^2}+\expec{\left\|\frac1b\sum_{i=1}^{b}\frac{\xi_i-\xi_i'}{\nu}u_i\right\|^2}\\
	+& 2\expec{\left\langle\left(\frac1b\sum_{i=1}^{b}\frac{f(x+\nu u_i)-f(\theta)}{\nu}u_i-\nabla f_\nu(\theta)\right),\left(\frac1b\sum_{i=1}^{b}\frac{\xi_i-\xi_i'}{\nu}u_i\right)\right\rangle}.
	\end{align*}
	Now note that, using independence of $\xi_i,\xi_i'$, and $u_i$, we have $\forall~i$
	\begin{align*}
	&\expec{\left\langle\left(\frac{f(x+\nu u_i)-f(\theta)}{\nu}u_i-\nabla f_\nu(\theta)\right),\left(\frac{\xi_i-\xi_i'}{\nu}u_i\right)\right\rangle}\\
	=&\expec{\left\langle\frac{f(x+\nu u_i)-f(\theta)}{\nu}u_i-\nabla f_\nu(\theta),u_i\right\rangle}\expec{\frac{\xi_i-\xi_i'}{\nu}}=0
	\end{align*}
	We also have, $\forall~i\neq j$,
	\begin{align*}
	&\expec{\left\langle\left(\frac{f(x+\nu u_i)-f(\theta)}{\nu}u_i-\nabla f_\nu(\theta)\right),\left(\frac{\xi_j-\xi_j'}{\nu}u_j\right)\right\rangle}\\
	=&\expec{\left\langle\frac{f(x+\nu u_i)-f(\theta)}{\nu}u_i-\nabla f_\nu(\theta),u_j\right\rangle}\expec{\frac{\xi_j-\xi_j'}{\nu}}=0\numberthis\label{eq:crossprod}
	\end{align*}
	Using, Lemma~\ref{lm:zogradvar}, we hence have,	
	\begin{align*}
	\expec{\left\|\frac1b\sum_{i=1}^{b}\frac{f(x+\nu u_i)-f(\theta)}{\nu}u_i-\nabla f_\nu(\theta)\right\|^2}\leq \frac{2(d+5)\|\nabla f(\theta)\|^2}{b}+\frac{\nu^2 M^2(d+3)^3}{2b}.\numberthis\label{eq:truegraderror}
	\end{align*}
	Furthermore, we have
	\begin{align*}
	\expec{\left\|\frac1b\sum_{i=1}^{b}\frac{\xi_i-\xi_i'}{\nu}u_i\right\|^2}=\frac{1}{b^2}\sum_{i=1}^{b}\expec{\frac{(\xi_i-\xi_i')^2}{\nu^2}}\expec{\|u_i\|^2}=\frac{2d\sigma^2}{b\nu^2}.\numberthis\label{eq:noisevar}
	\end{align*}
	Combining, \eqref{eq:truegraderror}, \eqref{eq:crossprod}, and \eqref{eq:noisevar}, we obtain Lemma~\ref{lm:zogradesterrordiffnoise}.	
\end{proof}

\begin{proof}[of Theorem~\ref{thm:zlmcdiffnoise}]
	Using Lemma~\ref{lm:zogradesterrordiffnoise}, \eqref{eq:lmcW2recfinal} changes to,
	\begin{align*}
	W_2(\varpi_n,\pi)
	\leq & (1-0.5mh)^nW_2(\varpi_0,\pi)  + \frac{3.3M\sqrt{hd}}{m}+\frac{2\nu M\sqrt{d}}{m}+\frac{\nu M\sqrt{h}}{2\sqrt{mb}}(d+3)^\frac{3}{2}\\
	+&\frac{3\sqrt{h(d+5)(\frac{\sigma^2}{\nu^2}+2Md)}}{\sqrt{mb}}.
	\end{align*}
	Now the last term involves $\nu$ in the denominator. To counter the effect we have to increase the sample size to $b=\frac{d}{\epsilon^2}$.
\end{proof}

\begin{proof}[of Lemma~\ref{lemma:lipfunc}]
First note that, we have
	\begin{align*}
	g_{\nu,b}(\theta)-\nabla f_\nu(\theta)=&\frac1b\sum_{i=1}^{b}\frac{F(\theta+\nu u_i,\xi_i)-F(\theta,\xi_i')}{\nu}u_i(\theta)-\nabla f_\nu(\theta)\\
	=&\frac1b\sum_{i=1}^{b}\frac{F(\theta+\nu u_i,\xi_i)-F(\theta,\xi_i)}{\nu}u_i-\nabla f_\nu(\theta)+\frac1b\sum_{i=1}^{b}\frac{F(\theta,\xi_i)-F(\theta,\xi_i')}{\nu}u_i
	\end{align*}
	Hence, we have
	\begin{align*}
	\expec{\left\|g_{\nu,b}(\theta)-\nabla f_\nu(\theta)\right\|^2}=&2\expec{\left\|\frac1b\sum_{i=1}^{b}\frac{F(\theta+\nu u_i,\xi_i)-F(\theta,\xi_i)}{\nu}u_i-\nabla f_\nu(\theta)\right\|^2}\\ +&2\expec{\left\|\frac1b\sum_{i=1}^{b}\frac{F(\theta,\xi_i)-F(\theta,\xi_i')}{\nu}u_i\right\|^2}.
	\end{align*}

	Using, Lemma~\ref{lm:zogradvar}, we have,	
	\begin{align*}
	\expec{\left\|\frac1b\sum_{i=1}^{b}\frac{F(\theta+\nu u_i,\xi_i)-F(\theta,\xi_i)}{\nu}u_i-\nabla f_\nu(\theta)\right\|^2}\leq \frac{2(d+5)(\|\nabla f(\theta)\|^2+\sigma^2)}{b}+\frac{\nu^2 M^2(d+3)^3}{2b}.\numberthis\label{eq:truegraderrorgendiffnoise}
	\end{align*}
	Furthermore, note that
	\begin{align*}
	&\expec{\left\|\frac1b\sum_{i=1}^{b}\frac{F(\theta,\xi_i)-F(\theta,\xi_i')}{\nu}u_i\right\|^2}=\frac{1}{b^2}\sum_{i=1}^{b}\expec{\frac{(F(\theta,\xi_i)-F(\theta,\xi_i'))^2}{\nu^2}}\expec{\|u_i\|^2}\\
	\leq& \frac{L^2}{b^2}\sum_{i=1}^{b}\expec{\frac{(\xi_i-\xi_i')^2}{\nu^2}}\expec{\|u_i\|^2}=\frac{2dL^2\sigma^2}{b\nu^2}.\numberthis\label{eq:noisevargendiffnoise}
	\end{align*}
	Combining, \eqref{eq:truegraderrorgendiffnoise}, and \eqref{eq:noisevargendiffnoise}, we get the result stated in Lemma~\ref{lm:zogradesterrordiffnoise}.	
\end{proof} 


\section{Proofs for Section~\ref{sec:hdzlmc}}
\begin{proof}[of Theorem~\ref{thm:hdzlmc}]
First, we have,
\begin{align*}
\Pr\{\hat S\neq S^*\}
&= \Pr\{\max_{j\in D\setminus S^*}|[g_{\nu,b}]_j|>\tau \textrm{ or } \min_{j\in S^*}|[g_{\nu,b}]_j|<\tau\} \\
&\leq \Pr\{\max_{j\in D\setminus S^*}|[g_{\nu,b}]_j|>\tau\} + \Pr\{\min_{j\in S^*}|[g_{\nu,b}]_j|<\tau\} \\
&\leq \sum_{j\in D\setminus S^*}\Pr\{|\zeta_j|>\tau\} + \sum_{j\in S^*}\Pr\{|\zeta_j|>a'-\tau\},
\end{align*}
where $a'=a-M\nu\sqrt{s}\leq a-\|\nabla f(\theta)-\nabla f_\nu(\theta)\|$ is a lower bound for $|[\nabla f_\nu(\theta)]_j|$.
Next we utilize concentration inequalities to give a bound for the tail of approximation error $\zeta_j$. Denote $[g_{\nu,1}]_j = \frac{f(\theta+\nu u)-f(\theta)}{\nu}u_j \defeq \phi(\nu,u)u_j$, where $\phi(\nu,u)$ is sub-exponential with
\begin{align*}
\|\phi(\nu,u)\|_{\Psi_1}
&= \sup_{p\geq1} p^{-1}(\E[|\phi(\nu,u)|^p])^{1\slash p} \\
&\leq \sup_{p\geq1} p^{-1}(\E[|\frac{f(\theta+\nu u)-f(\theta)-\nabla f(\theta)^\top \nu u}{\nu}|^p])^{1\slash p} + \sup_{p\geq1} p^{-1}(\E[|\nabla f(\theta)^\top u|^p])^{1\slash p} \\
&\leq \frac12M\nu\sup_{p\geq1} p^{-1}(\E[\|u\|^{2p}])^{1\slash p} + \|\nabla f(\theta)\|\sup_{p\geq1} p^{-1}(\E[\|u\|^p])^{1\slash p} \\
&\leq M\nu\|u\|_{\Psi_2}^2 + \|\nabla f(\theta)\|\|u\|_{\Psi_2} \\
&\leq 2R\|u\|_{\Psi_2},
\end{align*}
where $\|\cdot\|_{\Psi_1} = \sup_{p\geq1}p^{-1}\E[|\cdot|^p]^{1\slash p}$ and $\|\cdot\|_{\Psi_2} = \sup_{p\geq1}p^{-1\slash2}\E[|\cdot|^p]^{1\slash p}$ are the sub-exponential and sub-Gaussian norm respectively (see, for example~\cite{vershynin2018high} for more details). In the last inequality we require that $\nu\leq\frac{R}{M\|u\|_{\Psi_2}}$. Note that $u\sim N(0,\boldsymbol{I}_d)$ can be replaced by $\sum_{k\in S^*}u_ke_k\sim N(0,\boldsymbol{I}_s)$ due to Assumption \ref{spars_assum}. Moreover, we have the following estimate.
\begin{align*}
\|u_1\|_{\Psi_2}
&\leq \inf\{c>0: \E\left[\exp\left\{\frac{u_1^2}{c^2}\right\}\right]\leq2\} = \sqrt{\frac83} \defeq C_1, \\
\|u\|_{\Psi_2}
&\leq \inf\{c>0: \E\left[\exp\left\{\frac{\|u\|^2}{c^2}\right\}\right]\leq2\} \\
&= \sqrt{\frac{2}{1-2^{-2\slash d}}} \\
&\leq \sqrt{\frac{d}{\log2(1-\log2)}} \defeq C_2\sqrt{d},
\end{align*}
which implies that $\|\phi(\nu,u)\|_{\Psi_1}\leq 2RC_2\sqrt{s},\;\|u_1\|_{\Psi_2}\leq C_1$. We now state the following concentration inequality proved in~\cite{balasubramanian2018tensor}.
\begin{lemma} 
    \label{lemma:sum_sub_exponential}
    Let $(X_i, Y_i)$, $i=1,\ldots, n$  be $n$ independent copies of random variables $X$ and $Y$. Let $X$ be a sub-Gaussian random variable with $\| X \|_{\psi_2} \leq \Upsilon_1$, and $Y$ be a sub-exponential random variable with   $\| Y\|_{\psi_1} \leq \Upsilon_2$ for some constants $\Upsilon_1$ and $\Upsilon_2$.  Then for any $t \geq K \cdot \max\{  \Upsilon_1 ^3, \Upsilon_1 \}   \cdot \Upsilon_2$, we have
\begin{align*}
     Pr \biggl \{ \bigg|\sum_{i=1}^n \big[X_i  \cdot Y_i - \E(X  Y)\big]\bigg|  \geq t \biggr \}   \leq 4 \exp \biggl \{  - K_1\cdot \min \biggl[ \biggl ( \frac{t}{\sqrt{n} \Upsilon_1  \cdot \Upsilon_2} \biggr)^{2}, \biggl ( \frac{t}{\Upsilon_1 \cdot \Upsilon_2} \biggr )^{2/3} \biggr] \biggr \},
    \end{align*}
       where $K$ and $K_1$ are absolute constants.
\end{lemma}
From Lemma~\ref{lemma:sum_sub_exponential}, for $n\geq\max\left\{K_1\frac{2RC\sqrt{s}}{\tau},\left(\frac{2RC\sqrt{s}}{\tau}\right)^4\right\}$, we have:
\begin{align*}
\Pr\{|\zeta_j|\geq\tau\}
&= \Pr\left\{\left|\frac1n\sum_{k=1}^ng_{\nu,1}^k-\E[g_{\nu,1}]\right|\geq\tau\right\} \\
&\leq 4\exp\left\{-K_2\left(\frac{n\tau}{\|\phi(\nu,u)\|_{\Psi_1}\|u_1\|_{\Psi_2}}\right)^{2\slash3}\right\} \\
&\leq 4\exp\left\{-K_2\left(\frac{n\tau}{2RC\sqrt{s}}\right)^{2\slash3}\right\},
\end{align*}
where $C=C_1C_2=\sqrt{\frac{8}{3\log2(1-\log2)}}, K_1, K_2$ are absolute constants. Therefore, by setting the threshold $\tau=a'\slash2$, the probability of error is bounded by
\begin{align*}
\Pr\{\hat S\neq S^*\}
&\leq \sum_{j\in D\setminus S^*}\Pr\{|\zeta_j|>\tau\} + \sum_{j\in S^*}\Pr\{|\zeta_j|>a'-\tau\} \\
&\leq 4(d-s)\exp\left\{-K_2\left(\frac{n\tau}{2RC\sqrt{s}}\right)^{2\slash3}\right\} + 4s\exp\left\{-K_2\left(\frac{n(a'-\tau)}{2RC\sqrt{s}}\right)^{2\slash3}\right\} \\
&= 4d\exp\left\{-K_2\left(\frac{n(a-M\nu\sqrt{s})}{4RC\sqrt{s}}\right)^{2\slash3}\right\}.
\end{align*}
Given a pre-specified error rate $\epsilon>0$, it suffices to have $\nu\leq\frac{a}{2M\sqrt{s}} \wedge \frac{R}{MC_2\sqrt{s}}$ and 
\begin{align*}
n \geq \frac{8RC\sqrt{s}}{a}\left(\frac1{K_2}\log\frac{4d}{\epsilon}\right)^{3\slash2} \vee K_1\frac{8RC\sqrt{s}}{a} \vee \left(\frac{8RC\sqrt{s}}{a}\right)^4.
\end{align*}
\end{proof}